\documentclass [a4paper, 12pt, reqno]{amsart}
\usepackage [latin1]{inputenc}
\usepackage {a4}
\usepackage{amscd}
\usepackage{epsfig}
\usepackage{amssymb}
\usepackage{amsmath}
\usepackage{amsthm}
\usepackage[T1]{fontenc}
\usepackage{ae,aecompl}

\newcommand{\R} {\ensuremath{\mathbb{R}}}
\newcommand{\N} {\ensuremath{\mathbb{N}}}
\newcommand{\C} {\ensuremath{\mathbb{C}}}
\newcommand{\Z} {\ensuremath{\mathbb{Z}}}

\renewcommand{\o}[1]{\overline{#1}}

\newcommand{\dq}{\overline{\partial}}
\newcommand{\wt}[1]{\widetilde{#1}}

\DeclareMathOperator{\Sing}{Sing}

\DeclareMathOperator{\Dom}{Dom}
\DeclareMathOperator{\md}{md}
\DeclareMathOperator{\Def}{Def}
\DeclareMathOperator{\Var}{Var}
\DeclareMathOperator{\id}{id}

\newtheorem {satz} {Satz} [section]
\newtheorem {lem} [satz] {Lemma}
\newtheorem {cor} [satz] {Corollary}
\newtheorem {defn} [satz] {Definition}

\newtheorem {thm} [satz] {Theorem}

\DeclareMathOperator{\dist}{dist}
\DeclareMathOperator{\supp}{supp}

\newcommand{\Ra}{\mathcal{R}}

\renewcommand{\theta}{\vartheta}

\title[$\dq$-Neumann on singular spaces] 
{Compactness of the $\dq$-Neumann operator on singular complex spaces}

\author{J. Ruppenthal}
\address{Department of Mathematics, University of Wuppertal, Gau{\ss}str. 20, 42119 Wuppertal, Germany.}
\email{ruppenthal@uni-wuppertal.de}

\date{\today}



\begin{document}

\begin{abstract} 
Let $X$ be a Hermitian complex space of pure dimension $n$.
We show that the $\dq$-Neumann operator on $(p,q)$-forms 
is compact at isolated singularities of $X$ if $p+q\neq n-1, n$ and $q\geq 1$.
The main step is the construction of compact solution operators
for the $\dq$-equation on such spaces which is based on a general
characterization of compactness in function spaces on singular spaces,
and that leads also to a criterion for compactness of more general
Green operators on singular spaces.
\end{abstract}

\maketitle

\section{Introduction}

The Cauchy-Riemann operator $\dq$ and the related $\dq$-Neumann operator play a central role in complex analysis.
Especially the $L^2$-theory for these operators is of particular importance and has become indispensable for the subject
after the fundamental work of
H\"ormander on $L^2$-estimates and existence theorems for the $\dq$-operator (see \cite{Hoe1} and \cite{Hoe2})
and the related work of Andreotti and Vesentini (see \cite{AnVe}).
By no means less important is Kohn's solution of the $\dq$-Neumann problem (see \cite{Ko1}, \cite{Ko2} and also \cite{KoNi}),
which implies existence and regularity results for the $\dq$-complex, as well (see Chapter III.1 in \cite{FK}).
Important applications of the $L^2$-theory are for instance the Ohsawa-Takegoshi extension theorem \cite{OT},
Siu's analyticity of the level sets of Lelong numbers \cite{Siu0}
or the invariance of plurigenera \cite{Siu}.

Whereas the theory of the $\dq$-operator and the $\dq$-Neumann operator is very well developed on
complex manifolds, not much is known about the situation on singular complex spaces which
appear naturally as the zero sets of holomorphic functions.
The further development of this theory is an important task
since analytic methods have led to fundamental advances in geometry on complex manifolds (see Siu's results mentioned above),
but these analytic tools are still missing on singular spaces.

After a first period of intensive research on the $L^2$-theory for the $\dq$-operator on singular spaces
(see \cite{Oh2}, \cite{P}, \cite{Hk}, \cite{N2}, \cite{PS1}, \cite{FoHa}),
there has been good progress in this subject recently due to
Pardon and Stern (see \cite{PS2}), Diederich, Forn{\ae}ss, {\O}vrelid and Vassiliadou
(see \cite{Fo}, \cite{DFV}, \cite{FOV2}, \cite{OV1}, \cite{OV2}),
Ruppenthal and Zeron (see \cite{Rp7}, \cite{Rp8}, \cite{RZ1}, \cite{RZ2}).
On the other hand, the $\dq$-Neumann operator has not been studied on singular complex spaces yet.
The purpose of the present paper is to initiate this branch of research in complex analysis
on singular complex spaces.

Let $X$ be a Hermitian complex space\footnote{A reduced complex space with a Hermitian metric on the regular part
which is induced by local embeddings into complex number space, hence extends continuously into the singular set.}
of pure dimension $n$ with isolated singularities only.
Our intention is to study the behavior of the $\dq$-Neumann operator in the presence of these singularities.

Let $\Omega\subset\subset X$ be a relatively compact domain, and assume that either $X$ is compact and $\Omega=X$,
or that $X$ is Stein and $\Omega$ has smooth strongly pseudoconvex boundary that does not contain singularities.
Let $\Omega^*=\Omega - \Sing X$ and $\dq_w$ the $\dq$-operator in the sense of distributions.\footnote{The $\dq$-operator
in the sense of distributions is the maximal closed $L^2$-extension of the $\dq$-operator.
The notation $\dq_w$ refers to this as a weak extension. We will also use the notation $\dq_s$ for the minimal (strong)
closed $L^2$-extension of the $\dq_{cpt}$-operator on smooth forms with compact support (see section \ref{sec:ds}).}
Then there are only finitely many obstructions to solvability of the $\dq_w$-equation in the $L^2$-sense
on $\Omega^*$ for forms of degree $(p,q)$ with $p+q\neq n$, $q\geq 1$.\footnote{
If $\Omega=X$ is compact, we also have to assume that $q=1$ in case $p+q=n+1$.
We keep this assumption throughout the text without mentioning it explicitly.}
This can be deduced from $L^2$-regularity results for the $\dq_w$-equation at isolated singularities due to
Forn{\ae}ss, {\O}vrelid and Vassiliadou \cite{FOV2} by use of Hironaka's resolution of singularities
(Theorem \ref{thm:fov2}).

Hence the $\dq$-operator in the sense of distributions
\begin{eqnarray*}
\dq_w: L^2_{p,q-1}(\Omega^*) \rightarrow L^2_{p,q}(\Omega^*)
\end{eqnarray*}
has closed range $\Ra(\dq_w)$ in $L^2_{p,q}(\Omega^*)$ for $p+q\neq n$.
So, the densely defined closed self-adjoint $\dq_w$-Laplacian
$$\Box=\dq_w\dq_w^* + \dq_w^* \dq_w$$
has closed range in $L^2_{p,q}(\Omega^*)$ for $p+q\neq n-1,n$,
and we obtain the orthogonal decomposition
$$L^2_{p,q}(\Omega^*) = \ker \Box_{p,q} \oplus \Ra(\Box_{p,q}).$$
Then the $\dq_w$-Neumann operator 
$$N_{p,q} = \Box_{p,q}^{-1}: L^2_{p,q}(\Omega^*) \rightarrow \Dom(\Box_{p,q}) \subset L^2_{p,q}(\Omega^*)$$
is well-defined by the following assignment: let $N_{p,q}u=0$ if $u\in \ker\Box_{p,q}$, and
$N_{p,q}u$ the uniquely defined preimage of $u$ orthogonal to $\ker \Box_{p,q}$ if $u\in\Ra(\Box_{p,q})$.
The main result of the present paper is compactness of this operator $N_{p,q}$:

\begin{thm}\label{thm:neumanncpt1}
Let $X$ be a Hermitian complex space of pure dimension $n$ with only isolated singularities,
and $\Omega\subset\subset X$ a relatively compact domain
such that either $\Omega=X$ is compact, or $X$ is Stein and
$\Omega$ has smooth strongly pseudoconvex boundary
that does not contain singularities.

Let $p+q\neq n-1, n$ and $q\geq 1$. If $\Omega=X$ is compact and $p+q=n+1$, let $q=1$.
Then the $\dq$-operator in the sense of distributions $\dq_w$
has closed range in $L^2_{p,q}(\Omega^*)$ and $L^2_{p,q+1}(\Omega^*)$
so that the $\dq_w$-Neumann operator
$$N_{p,q} = \Box_{p,q}^{-1}=(\dq_w\dq_w^*+\dq_w^*\dq_w)_{p,q}^{-1}: L^2_{p,q}(\Omega^*) \rightarrow \Dom\Box_{p,q}\subset L^2_{p,q}(\Omega^*)$$
is well-defined as above. $N_{p,q}$ is compact.
\end{thm}

We remark that this also implies compactness of the $\dq_s$-Neumann operator $N_{n-p,n-q}^s$
in degree $(n-p,n-q)$ under the same assumptions (see section \ref{sec:ds}).

\vspace{3mm}
Compactness of the $\dq$-Neumann operator is a classical topic of complex analysis on manifolds.
A classical approach 
(e.g. on compact manifolds or on strongly pseudoconvex domains in complex manifolds)
is to deduce compactness by the Rellich embedding theorem
from subelliptic estimates for the complex Laplacian (see \cite{S} for a recent comprehensive discussion of
the $\dq$-Neumann problem).

We choose a different approach to prove Theorem \ref{thm:neumanncpt1}.
It follows from the work of Forn{\ae}ss, {\O}vrelid and Vassiliadou \cite{FOV2} that there are solution
operators for the $\dq_w$-equation that have some gain of regularity at the isolated singularities
(see Theorem \ref{thm:fov2}). By use of a Riesz characterization theorem for precompactness on arbitrary
Hermitian manifolds (Theorem \ref{thm:precompact}), we deduce that these operators are actually
compact solution operators:

\begin{thm}\label{thm:cptsolution}
Let $X$ be a Hermitian complex space of pure dimension $n$ with only isolated singularities,
and $\Omega\subset\subset X$ a relatively compact domain
such that either $\Omega=X$ is compact, or $X$ is Stein and
$\Omega$ has smooth strongly pseudoconvex boundary
that does not contain singularities.

Let $p+q\neq n$ and $q\geq 1$. If $\Omega=X$ is compact and $p+q=n+1$, let $q=1$.
Then the range $\mathcal{R}(\dq_w)_{p,q}$ of the $\dq$-operator in the sense of distributions
$$\dq_w: L^2_{p,q-1}(\Omega^*) \rightarrow (\ker\dq_w)_{p,q} \subset L^2_{p,q}(\Omega^*)$$
has finite codimension in $(\ker\dq_w)_{p,q}$, and there exists a compact operator
$$S: \mathcal{R}(\dq_w)_{p,q} \rightarrow L^2_{p,q-1}(\Omega^*)$$
such that $\dq_w Sf=f$.
\end{thm}

That requires also Kohn's subelliptic estimates and Hironaka's resolution of singularities
which we use to distinguish between the treatment of the isolated singularities on the one hand
and the strongly pseudoconvex boundary of the domain $\Omega$ on the other hand.
Compactness of the $\dq_w$-Neumann operator (i.e. Theorem \ref{thm:neumanncpt1})
follows then by an argument of Hefer and Lieb
since $N_{p,q}$ can be expressed in terms of the compact solution operators (see \cite{HL}).

As a byproduct, we also obtain the following characterization of compactness of the $\dq$-Neumann
operator on singular spaces with arbitrary singularities
(in the spirit of some recent work of Gansberger \cite{G} and Haslinger \cite{H} about compactness
of the $\dq$-Neumann operator on domains in $\C^n$):

\begin{thm}\label{thm:compact3}
Let $Z$ be a Hermitian complex space of pure dimension $n$, $X\subset Z$ an open Hermitian submanifold
and $\dq$ a closed $L^2$-extension of the $\dq_{cpt}$-operator on smooth forms with compact support in $X$,
for example $\dq=\dq_w$ the $\dq$-operator in the sense of distributions.
Let $0\leq p,q \leq n$.

Assume that $\dq$ has closed range in $L^2_{p,q}(X)$ and in $L^2_{p,q+1}(X)$.
Then
$$\Box_{p,q} = \dq_{p,q} \dq^*_{p,q+1} + \dq^*_{p,q}\dq_{p,q-1}$$
has closed range and the following conditions are equivalent:

(i) The $\dq$-Neumann operator $N_{p,q}=\Box^{-1}_{p,q}$ is compact.

(ii) For all $\epsilon>0$, there exists $\Omega\subset\subset X$ such that $\|u\|_{L^2_{p,q}(X-\Omega,\varphi)}<\epsilon$
for all
$$u\in \{u\in \Dom(\dq)\cap\Dom(\dq^*)\cap \Ra(\Box_{p,q}): \|\dq u\|^2_{L^2} + \|\dq^*u\|^2_{L^2}<1\}.$$

(iii) There exists  a smooth function $\psi\in C^\infty(X,\R)$, $\psi>0$, such that $\psi(z)\rightarrow \infty$ as $z\rightarrow bX$,
and
\begin{eqnarray*}
( \Box_{p,q} u,u)_{L^2} \geq \int_X \psi |u|^2 dV_X\ \ \mbox{ for all } u\in\Dom(\Box_{p,q})\cap\Ra(\Box_{p,q}).
\end{eqnarray*}
\end{thm}

The present paper is organized as follows.
In section \ref{sec:precompact}, we give a criterion for $L^2$-precompactness of bounded sets of differential forms 
on arbitrary Hermitian manifolds
in the spirit of the classical Riesz characterization (Theorem \ref{thm:precompact}).

This criterion is used to study compactness of general Green operators on singular spaces
with arbitrary singularities (Theorem \ref{thm:compact}) in section \ref{sec:compact}.
Theorem \ref{thm:compact3}
is an easy corollary from Theorem \ref{thm:compact} in the special case of the $\dq_w$-Neumann operator.

In section \ref{sec:sop}, we use the results of Forn{\ae}ss, {\O}vrelid and Vassiliadou to construct
compact solution operators for the $\dq_w$-equation which are then used to show compactness of the $\dq_w$-Neumann
operator by the method of Hefer and Lieb in section \ref{sec:cpt}.
Note that the proof of Theorem \ref{thm:cptsolution} is contained in the proof of Theorem \ref{thm:neumanncpt}.
Finally, we study the $\dq_s$-Neumann operator in section \ref{sec:ds}.

\vspace{7mm}
{\bf Acknowledgements.} 
The author thanks Klaus Gansberger and Dariush Ehsani
for very helpful discussions on the topic.

\newpage
\section{Precompactness on Hermitian manifolds}\label{sec:precompact}

Let $X$ be a Hermitian manifold. If $f$ is a differential form on $X$,
we denote by $|f|$ its pointwise norm. For a weight function $\varphi\in C^0(X)$,
we denote by $L^2_{p,q}(X,\varphi)$ the Hilbert space of $(p,q)$-forms such that
$$\|f\|_{L^2_{p,q}(X,\varphi)}^2:=\int_X |f|^2 e^{-\varphi} dV_X < \infty.$$
Note that we may take different weight functions for forms of different degree.
We assume that $X$ is connected. For two points $p,q\in X$, let $\dist_X(p,q)$
be the infimum of the length of curves connecting $p$ and $q$ in $X$.
Let $\Phi: X \rightarrow X$ be a diffeomorphism. Then we call
$$\md(\Phi) := \sup_{p\in X} \dist_X(p,\Phi(p))$$
the mapping distance of $\Phi$.
If $Y$ is another Hermitian manifold and
$\Phi: X\rightarrow Y$ differentiable, the pointwise norm of the tangential map $\Phi_*$
is defined by
$$|\Phi_*|(p) := \sup_{\substack{v\in T_pX\\ |v|=1}} |\Phi_*(v)|_{T_{\Phi(p)}Y}.$$
This leads to the sup-norm of $\Phi_*$:
$$\|\Phi_*\|_\infty :=\sup_{p\in X} |\Phi_*|(p).$$
We also need to measure how far $\Phi_*: TX \rightarrow TX$ is from the identity mapping on tangential vectors (if $\Phi: X\rightarrow X$).
As the total space $TX$ inherits the structure of a Hermitian manifold,
$\dist_{TX}$ is also well defined, and we set
\begin{eqnarray*}
\|\Phi_* - \id\|_\infty &=& \sup_{p\in X} \sup_{\substack{v\in T_pX\\ |v|=1}} \dist_{TX}(\Phi_* v,v).
\end{eqnarray*}

\begin{defn}\label{defn:Def}
Let $\Omega\subset X$ open.
We call a diffeomorphism 
$\Phi: (\Omega,X) \rightarrow (\Omega,X)$
a $\delta$-variation of $\Omega$ in $X$ if $\Phi|_{X-\Omega}$ is the identity map,
mapping distance $\md(\Phi)<\delta$ and $\|\Phi_*-\id\|_\infty , \|(\Phi^{-1})_*-\id\|_\infty < 3\delta$.
The set of all $\delta$-variations of $\Omega$ in $X$ will be denoted by
$\Var_\delta(\Omega,X)$.

A $\delta$-variation $\Phi\in \Var_\delta(\Omega,X)$ will be called $\delta$-deformation,
if it can be connected by a smooth path to the identity map in $\Var_\delta(\Omega,X)$,
i.e. if there exists a smooth map
$$\Phi_\cdot(\cdot): [0,1]\times X \rightarrow X,\ (t,x)\mapsto \Phi_t(x)\in X,$$
such that $\Phi_t(\cdot)\in \Var_\delta(\Omega,X)$ for all $t\in[0,1]$, $\Phi_0=\id$, $\Phi_1=\Phi$ and
\begin{eqnarray}\label{eq:dt}
\left|\frac{\partial}{\partial t} \Phi_t(x)\right| \leq 3\delta\ \ \mbox{ for all } t\in[0,1], x\in X.
\end{eqnarray}
The set of all $\delta$-deformations of $\Omega$ in $X$ will be denoted by
$$\Def_\delta(\Omega,X).$$
\end{defn}

\newpage
Note that $\|\Phi_*-\id\|_\infty, \|(\Phi^{-1})_* -\id\|_\infty < 3\delta$ implies that
\begin{eqnarray}\label{eq:3delta}
\|\Phi_*\|_\infty, \|(\Phi^{-1})_*\|_\infty < 1+3\delta.
\end{eqnarray}
A remark on condition \eqref{eq:dt} is in order:
if $\Phi$ is a $\delta$-deformation and $x\in X$, then $\Phi_\cdot(x): [0,1]\rightarrow X$
is a path connecting $x$ and $\Phi(x)$.
Since $\dist_X(x,\Phi(x))<\delta$, condition \eqref{eq:dt} means that the path $\Phi_\cdot(x)$ is not too far away from a geodesic
(of uniform velocity) connecting the two points if $\dist_X(x,\Phi(x))$ comes close to $\delta$.
Another useful observation is the following:

\begin{lem}\label{lem:trafo}
Let $M_1$ and $M_2$ be Hermitian manifolds,
$U_1\subset M_1$ and $U_2 \subset M_2$ open sets, and $\Gamma: U_1 \rightarrow U_2$ a diffeomorphism.
Let $\Omega_1 \subset\subset U_1$ be an open subset of $U_1$ and $\Omega_2=\Gamma(\Omega_1)\subset\subset U_2$.
For $\Phi\in \Def_\delta(\Omega_2, M_2)$, we define the pull-back as
$$\Gamma^\# \Phi = \Gamma^{-1} \circ \Phi\circ\Gamma: (\Omega_1, M_1) \rightarrow (\Omega_1, M_1).$$
Let $C_\Gamma := \max\{1,\|\Gamma_*\|_{\infty,\o{\Omega_1}}^2, \|\Gamma^{-1}_*\|_{\infty,\o{\Omega_2}}^2\}$. Then
$$\Gamma^\#\Phi \in \Def_{C_\Gamma \delta}(\Omega_1, M_1)$$
for all $\delta>0$ and all $\Phi\in \Def_\delta(\Omega_2, M_2)$.
\end{lem}

\begin{proof}
First off all, $\Gamma^{-1}\circ\Phi\circ\Gamma$ is only defined on $U_1$,
but as it is the identity mapping on $U_1-\Omega_1$,
$\Gamma^\#\Phi$ is well-defined as a map $M_1\rightarrow M_1$ if we extend it
as the identity mapping to $M_1-U_1$.

Let $\Phi\in \Var_\delta(\Omega_2,M_2)$.
So $\md(\Gamma^\#\Phi) < C_\Gamma\delta$ and
\begin{eqnarray*}
\|(\Gamma^\#\Phi)_* -\id\|_\infty = \| \Gamma^{-1}_* \circ ( \Phi_* -\id) \circ \Gamma_*\|_\infty \leq C_\Gamma 3 \delta.
\end{eqnarray*}
Moreover, if $\Phi\in\Def_\delta(\Omega_2, M_2)$, then clearly $(\Gamma^\# \Phi)_t = \Gamma^\# \Phi_t$
connects $\Gamma^\#\Phi$ to the identity mapping and $|\frac{\partial}{\partial t} (\Gamma^\# \Phi)_t| \leq C_\Gamma 3\delta$.
\end{proof}

Note that with the same constant $C_\Gamma>0$ also
$$(\Gamma^{-1})^\#: \Def_\delta(\Omega_1, M_1) \rightarrow \Def_{C_\Gamma\delta}(\Omega_2, M_2).$$

Using $\delta$-deformations on $X$, we can characterize precompact sets in the spaces of
square-integrable differential forms in the spirit of the classical Riesz
characterization (see e.g. \cite{Alt}, 2.15).
Before, we need some preliminary considerations:

\begin{lem}\label{lem:density}
The space of smooth forms with compact support $C^\infty_{(p,q),cpt}(X)$ is  dense in $L^2_{p,q}(X,\varphi)$.
\end{lem}

\begin{proof}
This follows by the usual mollifier method with a suitable partition of unity.
Let $f\in L^2_{p,q}(X,\varphi)$ and $\epsilon>0$. Then there exists a compact subset $K\subset\subset X$
such that $\|\widehat{f}-f\|_{L^2_{p,q}(X,\varphi)} < \epsilon$ if we denote by $\widehat{f}$ the trivial
extension of $f|_K$ to $X$ (see e.g. \cite{Alt}, A.1.16.2).
Now then, cover $K$ by finitely many coordinate charts $U_1, ..., U_N$, and let $\{\psi_j\}$
be a partition of unity for $\{U_j\}$. Then each $\psi_j \widehat{f}$ can be approximated
by convolution with a Dirac sequence (mollifier method): there exist $h_j\in C^\infty_{(p,q),cpt}(X)$
such that $\|\psi_j\widehat{f}-h_j\|_{L^2_{p,q}(X,\varphi)} < \epsilon/N$.

Letting $h:=\sum_{j=1}^N h_j \in C^\infty_{(p,q),cpt}(X)$, it follows that $\|f-h\|_{L^2_{p,q}(X,\varphi)}<2\epsilon$.
\end{proof}

\newpage
We observe that small deformations cannot disturb the $L^2$-norm arbitrarily:

\begin{lem}\label{lem:deformation1}
Let $f\in L^2_{p,q}(X,\varphi)$. Then: for all $\epsilon>0$ and all $\Omega\subset\subset X$,
there exists $\delta_f>0$ such that
$$\|\Phi^* f -f\|_{L^2_{p,q}(X,\varphi)}<\epsilon$$
for all $\Phi\in \Def_{\delta_f}(\Omega,X)$.
\end{lem}

\begin{proof}
By the previous Lemma \ref{lem:density}, we can choose a sequence $f_j\in C^\infty_{(p,q),cpt}(X)$ such that
$$\|f-f_j\|_{L^2_{p,q}(X,\varphi)} \rightarrow 0 \ \ \mbox{ as } \ \ j\rightarrow \infty.$$
Let $\epsilon>0$ and $\Omega\subset\subset X$.
For $\Phi\in \Def_\delta(\Omega,X)$, let $\Psi=\Phi^{-1}$ which is again in $\Def_\delta(\Omega,X)$ by Definition \ref{defn:Def}, 
and consider
\begin{eqnarray*}
\|\Phi^* ( f-f_j)\|_{L^2_{p,q}(X,\varphi)}^2 &=& \int_X | \Phi^* (f-f_j)|^2 e^{-\varphi} dV_X\\
&\leq& (1+3\delta)^{2(p+q)} \int_X \Phi^* |f-f_j|^2 e^{-\varphi} dV_X,
\end{eqnarray*}
since $f-f_j$ is a $(p,q)$-form and $\|\Phi_*\|_\infty < 1+3\delta$ by \eqref{eq:3delta}.

With $\|\Psi_*\|_\infty < 1+3\delta$ and $n=\dim_\C X$, it follows that
\begin{eqnarray*}
\|\Phi^* ( f-f_j)\|_{L^2_{p,q}(X,\varphi)}^2 &\leq& (1+3\delta)^{2p+2q} \int_X |f-f_j|^2 \Psi^* \big( e^{-\varphi} dV_X \big)\\
&\leq& (1+3\delta)^{2(p+q+n)} \int_X |f-f_j|^2 \Psi^* \big( e^{-\varphi}\big) dV_X\\
&\leq& (1+3\delta)^{2(p+q+n)} \sup_{z\in \Omega} |e^{\varphi(z)-\varphi(\Psi(z))}| \int_X |f-f_j|^2 e^{-\varphi} dV_X\\
&=& (1+3\delta)^{2(p+q+n)} \sup_{z\in \Omega} |e^{\varphi(z)-\varphi(\Psi(z))}| \cdot \|f-f_j\|_{L^2_{p,q}(X,\varphi)}^2.
\end{eqnarray*}
Since $\varphi\in C^0(X)$, it is uniformly continuous on $\overline{\Omega}$.
So, $\dist_X(z,\Psi(z))< \delta$ implies that there exists a $\delta_0>0$ such that
\begin{eqnarray}\label{eq:def1}
\|\Phi^* ( f-f_j)\|_{L^2_{p,q}(X,\varphi)} \leq 2 \|f-f_j\|_{L^2_{p,q}(X,\varphi)}
\end{eqnarray}
if $\delta<\delta_0$.
Since $f_j\in C^\infty_{(p,q),cpt}(X)$, mapping distance $\md(\Phi)<\delta$ and $\|\Phi_*-\id \|_\infty<3\delta$,
we also get (for fixed $j$ ) that
\begin{eqnarray}\label{eq:def2}
\sup_{z\in \Omega} e^{-\varphi(z)} |f_j - \Phi^* f_j|^2(z)\rightarrow 0 \ \ \mbox{ as } \ \ \delta\rightarrow 0.
\end{eqnarray}
So then, choose $f_j$ such that $\|f-f_j\|<\epsilon/4$.
It follows by \eqref{eq:def1} and \eqref{eq:def2} that
\begin{eqnarray*}
\|\Phi^* f -f\|
&\leq& \|f-f_j\| + \|\Phi^* f - \Phi^* f_j\| + \| f_j - \Phi^* f_j\|\\
&\leq& 3 \|f-f_j\| + \mbox{Vol}(\Omega)^{1/2} \sup_{z\in \Omega} e^{-\varphi(z)/2} |f_j - \Phi^* f_j|(z)\\
&\leq& 3\epsilon/4 + \epsilon/4 = \epsilon
\end{eqnarray*}
if $\delta<\delta_0$ is small enough, say $\delta<\delta_f$.
\end{proof}

\newpage
Now, precompact sets in $L^2_{p,q}(X,\varphi)$ can be characterized by:

\begin{thm}\label{thm:precompact}
Let $X$ be a Hermitian manifold and $\mathcal{A}$ a bounded subset of $L^2_{p,q}(X,\varphi)$.
Then $\mathcal{A}$ is precompact if and only if the following two conditions are fulfilled:

(i) for all $\epsilon>0$ and all $\Omega\subset\subset X$, there exists $\delta>0$ such that
$$\|\Phi^* f -f\|_{L^2_{p,q}(X,\varphi)}<\epsilon$$
for all $\Phi\in \Def_\delta(\Omega,X)$ and all $f\in\mathcal{A}$.

(ii) for all $\epsilon>0$, there exists $\Omega_\epsilon\subset\subset X$ such that
$$\|f\|_{L^2_{p,q}(X-\Omega_\epsilon,\varphi)}<\epsilon$$
for all $f\in\mathcal{A}$.
\end{thm}

\begin{proof}
First, assume that $\mathcal{A}$ is precompact. Let $\epsilon>0$.
By definition, there exists an integer $N_\epsilon$ and forms $f_1, ..., f_{N_\epsilon}$
such that
$$\mathcal{A} \subset \bigcup_{j=1}^{N_\epsilon} B_{\epsilon/4}(f_j)\ \ \mbox{ in }\ L^2_{p,q}(X,\varphi).$$
To show (ii), choose $\Omega_\epsilon\subset\subset X$
so big that
$$\|f_j\|_{L^2_{p,q}(X-\Omega_\epsilon,\varphi)} < \epsilon/2\ \ \mbox{ for }\ j=1, ..., N_\epsilon,$$
which is possible because we have to consider only finitely many forms simultaneously
(see e.g. \cite{Alt}, A.1.16.2). For an arbitrary $f\in\mathcal{A}$, there exists an index $j_0$
such that
\begin{eqnarray*}
\|f\|_{L^2_{p,q}(X-\Omega_\epsilon,\varphi)} \leq \|f-f_{j_0}\|_{L^2_{p,q}(X-\Omega_\epsilon,\varphi)} + \|f_{j_0}\|_{L^2_{p,q}(X-\Omega_\epsilon,\varphi)} < \epsilon/4 +\epsilon/2.
\end{eqnarray*}
So, property (ii) is valid.

To show (i), we proceed analogously. Let $\Omega\subset\subset X$.
For each of the (finitely many) $f_j$, there exists by Lemma \ref{lem:deformation1}
a $\delta_j>0$ depending on $f_j$ and $\epsilon$ such that
$$\|\Phi^* f_j - f_j\|_{L^2_{p,q}(X,\varphi)} < \epsilon/4$$
for all $\Phi\in \Def_{\delta_j}(\Omega,X)$. Set $\delta':=\min_{1\leq j\leq N_\epsilon} \{\delta_j\}$.
Now then, there exists for any $f\in\mathcal{A}$ an index $j_0$ such that
\begin{eqnarray*}
\|\Phi^* f -f \|_{L^2_{p,q}(X,\varphi)} &\leq& 
\|\Phi^* f_{j_0} - f_{j_0}\|_{L^2_{p,q}(X,\varphi)} + \|f_{j_0} - f\|_{L^2_{p,q}(X,\varphi)}\\
&&  + \|\Phi^*(f-f_{j_0})\|_{L^2_{p,q}(X,\varphi)}\\
&<& \epsilon/4 + \epsilon/4 + \|\Phi^*(f-f_{j_0})\|_{L^2_{p,q}(X,\varphi)}
\end{eqnarray*}
for all $\Phi\in \Def_{\delta'}(\Omega,X)$.
But
$$ \|\Phi^*(f-f_{j_0})\|_{L^2_{p,q}(X,\varphi)} \leq 2  \|f-f_{j_0}\|_{L^2_{p,q}(X,\varphi)} < 2 \epsilon/4$$
for all $\Phi\in \Def_\delta(\Omega,X)$ (independent of $f$, $f_{j_0}$) if we choose $\delta<\delta'$ small enough
as in the proof of Lemma \ref{lem:deformation1} (see estimate \eqref{eq:def1}).
That proves property (i).\\

Conversely, assume that condition (i) and (ii) are satisfied.
Our first objective is to show that under these circumstances the approximation
by smooth compactly supported forms as in Lemma \ref{lem:density}
can be made uniformly for all $f \in\mathcal{A}$, i.e.
we will construct a sequence of operators
$${\bf T}_k: \mathcal{A} \rightarrow C^\infty_{(p,q),cpt}(X)\ , \ k\in\N,$$
such that
\begin{eqnarray}\label{eq:uniform}
\|{\bf T}_k f - f\|_{L^2_{p,q}(X,\varphi)} < 1/k\ \ \mbox{ for all }\ f\in \mathcal{A}.
\end{eqnarray}
We start by using property (ii) to choose an exhaustion
$$\Omega_1 \subset\subset \Omega_2 \subset \subset \Omega_3 \subset\subset \cdots \subset\subset X$$
of $X$ by open relatively compact subsets $\Omega_k$ such that
\begin{eqnarray}\label{eq:exh1}
\|f\|_{L^2_{p,q}(X-\Omega_k,\varphi)} < 1/(3(k+1))\ \ \mbox{ for all }\ f\in \mathcal{A}.
\end{eqnarray}

We will now define ${\bf T}_k$. Let $\chi:=\chi_{\Omega_k}$ be the characteristic function of $\Omega_k$.
Cover $\o{\Omega_k}$ by finitely many open sets $U_1, ..., U_N \subset \subset \Omega_{k+1}$ which are contained in coordinate charts,
and choose cut-off functions $\psi_j\in C^\infty_{cpt}(U_j)$, $0\leq \psi_j \leq 1$, such that
\begin{eqnarray}\label{eq:partition1}
\psi:=\sum_{j=1}^N \psi_j \in C^\infty_{cpt}(\Omega_{k+1})
\end{eqnarray}
satisfies
\begin{eqnarray}\label{eq:partition2}
0\leq \psi \leq 1\ \ \mbox{ and }\ \ \psi|_{\o{\Omega_k}} \equiv 1,
\end{eqnarray}
i.e. $\{\psi_j\}_j$ is a partition of unity on $\o{\Omega_k}$ (subordinate to $\{U_j\}$).
For $f\in\mathcal{A}$, let
$$f_j:=\psi_j f.$$
Note that the $f_j$ have compact support in $\Omega_{k+1}$ and that
\begin{eqnarray*}
\chi f = \chi \sum_{j=1}^N  f_j
\end{eqnarray*}
because $\sum \psi_j =\psi\equiv 1$ on $\o{\Omega_k}$ and $\chi$ is the characteristic function of $\Omega_k$.

Since $\psi_j\in C^\infty_{cpt}(X)$, we also observe that the forms $f_j=\psi_j f$
still satisfy condition (i) as $f$ runs through $\mathcal{A}$:
for all $\epsilon>0$ and all $\Omega\subset\subset X$, there exists $\delta_j>0$ such that
\begin{eqnarray}\label{eq:fj}
\|\Phi^* f_j -f_j\|_{L^2_{p,q}(X,\varphi)} = \|\Phi^*(\psi_j f) - \psi_j f\|_{L^2_{p,q}(X,\varphi)}<\epsilon
\end{eqnarray}
for all $\Phi\in \Def_{\delta_j}(\Omega,X)$ and all $f\in\mathcal{A}$.
The reason is that the factor $\psi_j$ can be absorbed exactly as the factor $e^{-\varphi}$
in the proof of Lemma \ref{lem:deformation1} (see the derivation of \eqref{eq:def1}).

\newpage
Since $U_j$ is contained in a coordinate chart, we will treat it (for simplicity of notation)
as an open set in $\C^n$. This is possible by Lemma \ref{lem:trafo}.
We will approximate $f_j$ on $U_j$ by smoothing with a Dirac sequence.
So, let $\eta\in C^\infty(\C^n)$, $0\leq \eta\leq 1$, with support in the unit ball, $\int \eta dV=1$,
and set
$$h_\epsilon := \epsilon^{-2n} h(z/\epsilon)\ \ \mbox{ for }\ \epsilon>0.$$

Now then, define
\begin{eqnarray*}
{\bf T}^j_k f (z) := (\chi f_j) * h_{\epsilon_j} (z) = \int_{U_j} h_{\epsilon_j}(z-\zeta) (\chi \psi_j f)(\zeta) dV_{\C^n}
\end{eqnarray*}
where $\epsilon_j>0$ has to be chosen later on (small enough), but at least
$$\epsilon_j < \dist_{\C^n}(\supp \psi_j, b U_j),$$
which also implies that $\epsilon_j < \dist_{\C^n}(\supp f_j, bU_j) \leq \dist_{\C^n}(\supp \chi f_j,bU_j)$.
Hence ${\bf T}^j_k f \in C^\infty_{(p,q),cpt}(U_j)$, and we can define 
\begin{eqnarray*}
{\bf T}_k f := \sum_{j=1}^N {\bf T}_k^j f = \sum_{j=1}^N (\chi \psi_j f) * h_{\epsilon_j} \in C^\infty_{(p,q),cpt}(X).
\end{eqnarray*}

Recall that $\chi=\chi_{\o{\Omega_k}}$, the covering $U_1, ..., U_N$, the choice of local coordinates
and the $\epsilon_j$ clearly depend on $k$. 

We will now show that
$$\|{\bf T}_k f - f\|_{L^2_{p,q}(X,\varphi)} < 1/k\ \ \mbox{ for all }\ f\in \mathcal{A}$$
if we choose the $\epsilon_j$ small enough.
By \eqref{eq:partition1}, \eqref{eq:partition2} and \eqref{eq:exh1},
we get
\begin{eqnarray*}
\|{\bf T}_k f -f\|_{L^2_{p,q}(X,\varphi)} &\leq& 
\|{\bf T}_k f -\psi f\|_{L^2_{p,q}(X,\varphi)} + \|\psi f - f\|_{L^2_{p,q}(X,\varphi)}\\
&=& \|{\bf T}_k f -\psi f\|_{L^2_{p,q}(X,\varphi)} + \|(1-\psi) f \|_{L^2_{p,q}(X-\Omega_{k-1},\varphi)}\\
&\leq& \|{\bf T}_k f -\psi f\|_{L^2_{p,q}(X,\varphi)} + \|f \|_{L^2_{p,q}(X-\Omega_{k-1},\varphi)}\\
&\leq& \|{\bf T}_k f -\psi f\|_{L^2_{p,q}(X,\varphi)} + 1/(3k)\\
&\leq& \sum_{j=1}^N \|{\bf T}_k^j f - \psi_j f\|_{L^2_{p,q}(X,\varphi)} + 1/(3k).
\end{eqnarray*}
It remains to show that $\sum \|{\bf T}^j_k f -\psi_j f\|_{L^2_{p,q}(X,\varphi)} \leq 2/(3k)$ for all $f\in\mathcal{A}$
if we choose the $\epsilon_j>0$ small enough. So, consider
\begin{eqnarray*}
{\bf T}_k^j f (z) - (\psi_j f)(z) &=& \int_{U_j} h_{\epsilon_j}(z-\zeta) (\chi\psi_j f)(\zeta) dV_{\C^n}(\zeta)\\
&& - (\psi_j f)(z) \int_{U_j} h_{\epsilon_j}(z-\zeta) \big[\chi(\zeta) + (1-\chi)(\zeta)\big] dV_{\C^n}(\zeta)\\
&=& \int_{U_j} h_{\epsilon_j}(z-\zeta) \chi(\zeta) \big[ (\psi_j f)(\zeta) - (\psi_j f)(z)\big] dV_{\C^n}(\zeta)\\
&& - (\psi_j f)(z) \int_{U_j} h_{\epsilon_j}(z-\zeta) (1-\chi)(\zeta) dV_{\C^n}(\zeta)\\
&=:& \Delta^j_1(z) - \Delta^j_2(z).
\end{eqnarray*}
Since $(1-\chi)(\zeta)=0$ for $\zeta\in \Omega_k$, we can choose $\epsilon_j$ (at least) so small
that the integral in $\Delta^j_2(z)$ vanishes if $z\in \Omega_{k-1}$. It follows that
\begin{eqnarray*}
|\Delta^j_2(z)| \leq |(\psi_j f)(z)| \ \ \mbox{ and } \ \ \Delta^j_2(z)=0 \mbox{ if } z\in \Omega_{k-1}.
\end{eqnarray*}
This yields by use of \eqref{eq:exh1} that
\begin{eqnarray*}
\sum_{j=1}^N \|{\bf T}^j_k f -\psi_j f\|_{L^2_{p,q}(X,\varphi)}
&\leq& \sum_{j=1}^N \|\Delta^j_1 \|_{L^2_{p,q}(X,\varphi)} + \sum_{j=1}^N \| \psi_j f\|_{L^2_{p,q}(X-\Omega_{k-1},\varphi)}\\
&\leq& \sum_{j=1}^N \|\Delta^j_1 \|_{L^2_{p,q}(X,\varphi)} + \| f\|_{L^2_{p,q}(X-\Omega_{k-1},\varphi)}\\
&<&  \sum_{j=1}^N \|\Delta^j_1 \|_{L^2_{p,q}(X,\varphi)} + 1/(3k).
\end{eqnarray*}
So, it only remains to show that $\|\Delta^j_1\|_{L^2_{p,q}(X,\varphi)} \leq 1/(3kN)$ for all $f\in\mathcal{A}$
if we choose $\epsilon_j>0$ small enough. Note that we already arranged $\epsilon_j$ so small that $\Delta^j_1$
has support in $U_j$.

By standard estimates for convolution integrals (see \cite{Alt}, 2.12.1,
as it is applied in the proof of \cite{Alt}, 2.15) and $|\chi|\leq 1$,
\begin{eqnarray*}
\|\Delta^j_1\|_{L^2_{p,q}(U_j,\varphi)}
&\lesssim& \| h_{\epsilon_j}\|_{L^1(\C^n)} \sup_{\substack{v\in\C^n\\|v|\leq \epsilon_j}} \| f_j(\cdot - v) - f_j\|_{L^2_{p,q}(U_j,\varphi)}.
\end{eqnarray*}
But translations by $v$ in the coordinate chart $U_j$ with $|v|\leq \epsilon_j$ can
be extended to some $\delta_j$-deformation of $\Omega_{k+1}$ in $X$ if $\epsilon_j$ is small enough,
because the connecting curves $\Phi_t(z)=z-tv$, $0\leq t\leq 1$, behave more and more like geodesics
as $v\rightarrow 0$ and $\o{\Omega_{k+1}}$ is compact. That shows that \eqref{eq:dt} is fulfilled if $v$ is small enough,
the other conditions from Definition \ref{defn:Def} are easy to check.
With Lemma \ref{lem:trafo}, we can assume that $\delta_j\rightarrow 0$ as $\epsilon_j\rightarrow 0$.

Hence
\begin{eqnarray*}
\|\Delta^j_1\|_{L^2_{p,q}(U_j,\varphi)}
&\lesssim& \sup_{\Phi \in \Def_{\delta_j}(\Omega_{k+1},X)} \|\Phi^* f_j -f_j\|_{L^2_{p,q}(X,\varphi)} < 1/(3kN)
\end{eqnarray*}
for all $f\in\mathcal{A}$ by \eqref{eq:fj} if we choose first $\delta_j$ and then $\epsilon_j$ small enough.
This completes the proof of \eqref{eq:uniform}, i.e. the operators ${\bf T}_k$ give a uniform
approximation of all $f\in\mathcal{A}$ by smooth forms with compact support in $\Omega_{k+1}$.

Since $\mathcal{A}$ is a bounded subset of $L^2_{p,q}(X,\varphi)$,
there exists a constant $C_k>0$ such that
\begin{eqnarray*}
|{\bf T}_k^j f (z)| &\leq& \|h_{\epsilon_j}\|_{L^2(\C^n)} \| f_j \|_{L^2_{p,q}(\C^n)} \leq C_k,
\end{eqnarray*}
and
\begin{eqnarray*}
|P {\bf T}_k^j f (z)| &\lesssim& \|\nabla h_{\epsilon_j}\|_{L^2(\C^n)} \| f_j \|_{L^2_{p,q}(\C^n)} \leq C_k
\end{eqnarray*}
for any first order differential operator with constant coefficients $P$.
It follows that
$$\mathcal{A}_k:=\{{\bf T}_k f: f\in\mathcal{A}\}$$
is a bounded subset of $C^1_{p,q}(\o{\Omega_{k+1}})$. 

Let $\gamma>0$.
By the Arzela-Ascoli Theorem, there exist finitely many forms
$g_l\in C^0_{p,q}(\o{\Omega_{k+1}})$, $l=1, ..., N(k,\gamma)$, such that
$$\mathcal{A}_k \subset \bigcup_{l=1}^{N(k,\gamma)} B_\gamma(g_l)$$
with respect to the $C^0$-norm in $C^0_{p,q}(\o{\Omega_{k+1}})$.
But there exists a constant $D_k>0$ such that
$$\|h\|_{L^2_{p,q}(\o{\Omega_{k+1}},\varphi)} \leq D_k \|h\|_{C^0_{p,q}(\o{\Omega_{k+1}})}$$
for all $(p,q)$-forms $h$.

Taking into account that the forms ${\bf T}_k f$ have compact support in $\Omega_{k+1}$,
it follows that
$$\mathcal{A}_k \subset \bigcup_{l=1}^{N(k,\gamma)} B_{\gamma D_k}(g_l)$$
in $L^2_{p,q}(X,\varphi)$ if we extend the forms $g_l$ trivially to $X$.
But $\|{\bf T}_k f-f\|_{L^2_{p,q}(X,\varphi)}<1/k$ for all $f\in\mathcal{A}$.
Hence
$$\mathcal{A} \subset \bigcup_{l=1}^{N(k,\gamma)} B_{\gamma D_k+1/k}(g_l)$$
in $L^2_{p,q}(X,\varphi)$.
This means that $\mathcal{A}$ is precompact because $\gamma D_k +1/k$
can be made arbitrarily small (by choosing first $k$ big and then $\gamma$ small enough).
\end{proof}

We remark that the criterion carries over to $L^p$-forms, $0\leq p<\infty$, without further
difficulties.

\newpage
\section{Compactness of Green operators on Hermitian spaces}\label{sec:compact}

Let $X$ be a Hermitian manifold, $\varphi\in C^0(X)$ a weight function,
and
$$T: \Dom T \subset L^2_*(X,\varphi) \rightarrow L^2_*(X,\varphi)$$
a densely defined closed linear partial differential operator such that $T^2=0$.
We will always assume that the smooth compactly supported forms $C^\infty_{*,cpt}(X)$ are contained in the domain of such an operator.
The adjoint operator $T^*$ is also closed and densely defined, $T^{**}=T$, $(T^*)^2=0$ and
$$(\ker T)^\perp = \o{\Ra(T^*)}\ ,\ (\ker T^*)^\perp = \o{\Ra(T)},$$
where we denote by $\Ra(T)$, $\Ra(T^*)$ the range of $T$ and $T^*$, respectively.
Then we define $P: L^2_*(X,\varphi) \rightarrow L^2_*(X,\varphi)$ by
\begin{eqnarray*}
\Dom P &=& \{u\in\Dom T\cap\Dom T^*: Tu\in \Dom T^*, T^*u\in \Dom T\},\\
P &=& T^* T + T T^*.
\end{eqnarray*}
Then:

\begin{thm}\label{thm:self-adjoint}
$P$ is a densely defined closed self-adjoint operator, $(Pu,u)\geq 0$.
\end{thm}

\begin{proof}
We adopt we proof of Proposition V.5.7 in \cite{LM},
where the statement is proved in case $T=\dq_w$, the $\dq$-operator
in the sense of distributions. The proof in \cite{LM}
is more or less taken from \cite{FK}, Proposition 1.3.8, and is essentially due to Gaffney \cite{Gf}.

It is easy to see that $P$ is a densely defined closed operator 
with $(Pu,u)\geq 0$ for all $u\in\Dom P$,
we will show that $P$ is self-adjoint by checking that 
$$\Dom P=\Dom P^*.$$
We need the following Lemma of J. von Neumann as it is presented in \cite{LM}, Lemma V.5.10:
\begin{lem}\label{lem:neumann}
Let $A: V\rightarrow H$ be a closed densely defined operator on a Hilbert space $H$.
Set $R=\id + A^*A$, $S=\id + A A^*$,
\begin{eqnarray*}
\Dom(R) &=& \{x\in \Dom A: Ax\in\Dom A^*\},\\
\Dom(S) &=& \{x\in\Dom A^*: A^*x \in\Dom A\}.
\end{eqnarray*}
Then $R:\Dom(R)\rightarrow H$, $S:\Dom(S)\rightarrow H$ are linear bijective maps and $R^{-1}$, $S^{-1}: H\rightarrow H$
are continuous self-adjoint operators.
\end{lem}

Set
$$F=\id + P.$$
So, $F$ is a densely defined closed operator with
$$\|Fu\| \|u\| \geq (Fu,u) = \|u\|^2 + \|Tu\|^2 + \|T^* u\|^2$$
for all $u\in\Dom F=\Dom P$. Hence $\ker F=0$ and $\Ra(F)$ is closed.
By Lemma \ref{lem:neumann}, 
$$(\id + TT^*)^{-1}\ \ \mbox{ and }\ \ (\id+T^*T)^{-1}$$
are bounded self-adjoint operators.
So,
$$S:= (\id+TT^*)^{-1} + (\id+T^*T)^{-1} -\id$$
is also bounded and self-adjoint.

We will now show that $F$ is surjective and $S=F^{-1}$.
This implies that $F$ and $P$ are self-adjoint.
Consider
\begin{eqnarray}\label{eq:S1}
(\id + TT^*)^{-1} -\id &=& \left[ \id - (\id+TT^*)\right] (\id + TT^*)^{-1} = - TT^* (\id + TT^*)^{-1},\\
(\id + T^*T)^{-1} -\id &=& \left[ \id - (\id+T^*T)\right] (\id + T^*T)^{-1} = - T^*T (\id + T^*T)^{-1}.
\label{eq:S2}\end{eqnarray}
Therefore
\begin{eqnarray}\label{eq:S3}
\Ra\big((\id+TT^*)^{-1}\big) &\subset& \Dom TT^*,\\
\Ra\big((\id +T^*T)^{-1}\big) &\subset& \Dom T^*T,\label{eq:S4}
\end{eqnarray}
and \eqref{eq:S1} yields
$$S = (\id + T^*T)^{-1} - TT^* (\id + TT^*)^{-1}.$$
This implies with \eqref{eq:S4} and $T^2=0$ that
$$\Ra(S) \subset\Dom T^*T$$
and
$$T^* T S = T^* T (\id + T^* T)^{-1}.$$
By symmetry, \eqref{eq:S2}, \eqref{eq:S3} and $(T^*)^2=0$ give
$\Ra(S) \subset\Dom TT^*$
and
$$TT^* S = TT^* (\id + TT^*)^{-1}.$$
Summing up, $\Ra(S) \subset\Dom F$ and
\begin{eqnarray*}
FS &=& S + T^*T S + TT^* S\\
&=& (\id+TT^*)^{-1} + (\id+T^*T)^{-1} -\id + T^* T (\id + T^* T)^{-1} + TT^* (\id + TT^*)^{-1}\\
&=& (\id + T^*T) (\id + T^*T)^{-1} + (\id +TT^*)(\id + TT^*)^{-1} -\id = \id
\end{eqnarray*}
on $H$. So, $\Ra(F)=H$ and $S=F^{-1}$.
\end{proof}

\begin{cor}
$P$ induces the orthogonal decomposition
\begin{eqnarray*}
L^2_*(X,\varphi) &=& \ker P \oplus \o{\Ra(P)}\\
&=& \left( \ker T \cap \ker T^*\right) \oplus \o{\Ra(T)} \oplus \o{\Ra(T^*)}.
\end{eqnarray*}
\end{cor}

\begin{proof}
The first equality follows from the self-adjointness of $P$, $\ker P = \ker T\cap \ker T^*$ follows from
$(Pu,u)=\|Tu\|^2 + \|T^* u\|^2$. Clearly, $\o{\Ra(P)}\subset \o{\Ra(T)} + \o{\Ra(T^*)}$
and $\o{\Ra(T)}\perp\o{\Ra(T^*)}$ yield $\o{\Ra(P)} \subset \o{\Ra(T)} \oplus \o{\Ra(T^*)}$.
On the other hand, 
assume that $f \perp \o{\Ra(T)}\oplus\o{\Ra(T^*)}$. Then $(f,Tg)=0$, $(f,T^*h)=0$ for all $g\in\Dom T$, $h\in\Dom T^*$.
Thus $f\in \Dom T^*\cap\Dom T$ and $(T^*f,g)=(Tf,h)=0$ for all $g,h\in L^2_*(X,\varphi)$ since $T$ and $T^*$ are densely defined. 
Hence, $f\in (\ker T\cap\ker T^*)$.
\end{proof}

\newpage

\begin{lem}\label{lem:Q}
Let $V_1, V_2, V_3 \subset L^2_*(X,\varphi)$ be closed subspaces such that
\begin{eqnarray}\label{eq:range1}
\Ra(T|_{V_1}) \subset V_2,\  \Ra(T|_{V_2})\subset V_3,
\end{eqnarray}
these ranges are closed, and $T|_{V_1}^* = T^*|_{V_2}$, $T|_{V_2}^* = T^*|_{V_3}$. 
It follows that the densely defined closed restricted operator
$$Q=P|_{V_2}: V_2 \rightarrow V_2$$
is self-adjoint and has closed range. Hence,
\begin{eqnarray}\label{eq:decomp1}
V_2 &=& \ker Q \oplus \Ra(Q)\\\label{eq:decomp2}
&=& \left(\ker T|_{V_2} \cap \ker T^*|_{V_2}\right) \oplus \Ra(T|_{V_1}) \oplus \Ra(T^*|_{V_3}),
\end{eqnarray}
and there exists a constant $c>0$ such that
\begin{eqnarray}\label{eq:apriori}
\begin{array}{llll}
c\|u\|^2 & \leq& \| Tu\|^2 + \|T^*u\|^2,\ \ \ &  \mbox{ for } u\in \Dom(T)\cap \Dom(T^*)\cap \Ra(Q),\\
c\|u\| & \leq& \| Qu\|,  & \mbox{ for } u\in \Dom(Q)\cap \Ra(Q).
\end{array}
\end{eqnarray}
\end{lem}

\begin{proof}
By assumption,
\begin{eqnarray}\label{eq:range2}
\Ra(T^*|_{V_2}) = \Ra(T|_{V_1}^*) \subset V_1,\ \Ra(T^*|_{V_3})=\Ra(T|_{V_2}^*)\subset V_2,
\end{eqnarray}
and all these ranges are closed (see e.g. \cite{KK}, Proposition A.1.2).
\eqref{eq:range1} and \eqref{eq:range2} together imply that
\begin{eqnarray}\label{eq:range3}
\Ra(P|_{V_2}) \subset V_2.
\end{eqnarray}
It is clear that $Q=P|_{V_2}: V_2 \rightarrow V_2$ is closed and densely defined.

Let $V_2^\perp$ be the orthogonal complement of $V_2$ in $L^2_*(X,\varphi)$
and $u\in V_2^\perp\cap\Dom P$. Then
$$(Pu,v)=(u,Pv)=0$$
for all $v\in V_2\cap\Dom P$. As $\Dom P$ is dense in $V_2$,
this yields
$$\Ra(P|_{V_2^\perp}) \subset V_2^\perp.$$
It follows with \eqref{eq:range3} that
$$P=P|_{V_2}\oplus P|_{V_2^\perp}: V_2\oplus V_2^\perp \rightarrow V_2\oplus V_2^\perp.$$
Since $P$ is self-adjoint, it follows that both operators $Q=P|_{V_2}$ and $P|_{V_2^\perp}$ are self-adjoint.

Since $\Ra(T|_{V_1})$ is a closed subspace of $V_2$, we get the orthogonal decomposition
$$V_2 = \ker T|_{V_1}^* \oplus \Ra(T|_{V_1}).$$
On the other hand,
$$\Ra(T|_{V_2}^*) = \Ra(T^*|_{V_3}) \subset \ker T^*|_{V_2} = \ker T|_{V_1}^*$$
implies the orthogonal decomposition
$$\ker T_{V_1}^* = \left(\ker T_{V_1}^* \cap \ker T|_{V_2}\right) \oplus \Ra(T|_{V_2}^*).$$
Together, we obtain \eqref{eq:decomp2}. Since $Q$ is self-adjoint, we also have the
orthogonal decomposition $V_2= \ker Q\oplus\o{\Ra(Q)}$. Hence
$$\o{\Ra(Q)} = \Ra(T|_{V_1}) \oplus \Ra(T^*|_{V_3}).$$
To show that the range of $Q$ is closed, let $u\in\Ra(T|_{V_1})$.
Then $u = Tf$ with $f\in (\ker T|_{V_1})^\perp = \Ra(T^*|_{V_2})$.
So, $f=T^* g$ with $g\in (\ker T^*|_{V_2})^\perp = \Ra(T|_{V_1})$.
Hence, $g\in\Dom Q$ and
$$u = TT^* g = Qg.$$
Analogously, if $u\in\Ra(T^*|_{V_3})$, then there exists $g\in\Dom Q$
such that
$$u = T^* Tg=Qg.$$
This shows that $Q$ has closed range and \eqref{eq:decomp1} holds.

To prove the two estimates, we follow \cite{LM}, Theorem V.6.2.
First, We construct bounded solution operators for $T$ and $T^*$.
We elaborate that for $T$, the case of $T^*$ is analogous.
Let
$$L=\{u\in \Dom T|_{V_2}: u \perp \ker T\}$$
be the Banach space with the norm
$$\|u\|_L^2 := \|u\|^2 + \|Tu\|^2.$$
So, the mapping
$$A: L \rightarrow \Ra(T|_{V_2}),\ u\mapsto Tu,$$
is a bounded linear isomorphism. Therefore,
$$A^{-1}: \Ra(T|_{V_2}) \rightarrow L$$
is an $L^2$-bounded solution operator for $T$. Analogously,
let
$$B^{-1}: \Ra(T^*|_{V_2}) \rightarrow \{u\in\Dom T^*|_{V_2}: u\perp \ker T^* \}$$
be the corresponding $L^2$-bounded solution operator for $T^*$.

So, let
$u\in\Dom T\cap\Dom T^*\cap \Ra(Q)$.
Then
$$u=u_1 + u_2 \in \Ra(T|_{V_1})\oplus \Ra(T^*|_{V_3}).$$
Clearly,
$$u_1\in \Dom T^*,\ u_2 \in \Dom T,\ T^* u_1=T^* u,\ T u_2= Tu.$$
By our previous considerations,
$$u_1 = B^{-1} T^* u\ \ \mbox{ and }\ \ u_2=A^{-1} T u.$$
The continuity of $A^{-1}$ and $B^{-1}$ yields that
$$\|u\|^2 = \|u_1\|^2 + \|u_2\|^2 \leq C\left(\|T^*u\|^2 + \|Tu\|^2\right)$$
with a constant $C>0$ independent of $u$.

For $u\in \Dom Q\cap \Ra(Q)$, the second inequality follows easily:
\begin{eqnarray*}
C^{-1} \|u\|^2 \leq \|T^*u\|^2 + \|Tu\|^2 = (Qu,u) \leq \|Qu\| \|u\|.
\end{eqnarray*}
\end{proof}

\newpage

Under the assumptions of Lemma \ref{lem:Q}
we can construct the Green operator to $Q$ analogously to the construction
of the solution operators for $T$ and $T^*$ in the proof of Lemma \ref{lem:Q}.
Let
$$L = \{u\in\Dom Q: u\perp \ker Q\} = \{u\in\Dom Q: u\in \Ra(Q)\}$$
be the Banach space with the norm
$$\|u\|_L^2 := \|u\|^2 + \|Qu\|^2.$$
Note that $\|u\|_L \lesssim \|Qu\|$ by \eqref{eq:apriori}.
So, the mapping
$$Q|_L: L \rightarrow \Ra(Q)$$
is a bounded linear isomorphism. Hence,
$$Q|_L^{-1}: \Ra(Q) \rightarrow L \subset \Dom(Q)$$
is an $L^2$-bounded solution operator for $Q$.
We extend $Q|_L^{-1}$ to an operator
\begin{eqnarray}\label{eq:Green}
Q^{-1}: V_2 \rightarrow \Dom(Q)
\end{eqnarray}
by setting $Q^{-1}u=0$ for $u\in\ker Q$. $Q^{-1}$ is called the {\bf Green operator} associated to $Q$.
The main objective of the present section is to study necessary 
and sufficient conditions for compactness of $Q^{-1}$.

We need another useful representation of $Q^{-1}$ which goes back to E. Straube
in case of the complex Green operator (i.e. the $\dq$-Neumann operator)
on pseudoconvex domains in $\C^n$ (see \cite{S}, Theorem 2.9).

From now on, let
$$D := \{u\in V_2: u\in\Dom T\cap\Dom T^*\cap\Ra(Q)\}$$
be the Hilbert space with 
$$(u,v)_D := (Tu,Tv)_{L^2} + (T^*u,T^*v)_{L^2},$$
and
$$j: D \hookrightarrow V_2$$
the injection into $V_2$ which is bounded by \eqref{eq:apriori}. Let
$$j^*: V_2 \rightarrow D$$
be the adjoint operator. Then $Q^{-1}= j\circ j^*$ as we will show now.
Let 
$$u=u_1 + u_2 \in V_2 = \ker Q \oplus \Ra(Q).$$
Then:
$$(u,v)_{L^2} = (u,jv)_{L^2} = (j^* u,v)_D$$
for all $v\in D$. On the other hand,
$$(u,v)_{L^2} = (u_2,v)_{L^2} = (QQ^{-1} u_2,v)_{L^2} = (Q^{-1}u_2,v)_D = (Q^{-1} u,v)_D$$
for all $v\in D$. Hence, if $Q^{-1}$ is interpreted as an operator to $D$, then $Q^{-1}=j^*$.
It follows that
\begin{eqnarray}\label{eq:Q-1}
Q^{-1}=j\circ j^*: V_2 \rightarrow V_2.
\end{eqnarray}

\newpage

We will now characterize compactness of $Q^{-1}$ under the assumption that
$T\oplus T^*$ is elliptic
in the interior of $X$ in the sense that the G\r{a}rding inequality holds on
relatively compact subsets of $X$: for each bounded open subset $\Omega\subset\subset X$
there exists a constant $C_\Omega>0$ such that
\begin{eqnarray}\label{eq:garding1}
\|u\|^2_{W^{1,2}_*(\Omega,\varphi)} \leq C_{\Omega} \left( \|u\|^2_{L^2_*(\Omega,\varphi)} + \|Tu\|^2_{L^2_*(\Omega,\varphi)} + \|T^* u\|^2_{L^2_*(\Omega,\varphi)}\right)
\end{eqnarray}
for all $u\in C^\infty_{*,cpt}(\Omega)$. The Sobolev-norm $W^{1,2}$ is well defined on $\Omega$ for $\Omega\subset\subset X$.

Natural choices for $T$ are closed extensions of the operators $\dq_{cpt}$, $\partial_{cpt}$ or $d_{cpt}$ 
acting on smooth forms with compact support in $X$,
for example $T=\dq_w$, the $\dq$-operator in the sense of distributions (the maximal closed $L^2$-extension of $\dq_{cpt}$),
or $T=\dq_s$, the $\dq$-operator in the sense of approximation
by smooth forms with compact support (the minimal closed $L^2$-extension of $\dq_{cpt}$). 
In both cases, it is well-known that the G\r{a}rding
inequality \eqref{eq:garding1} holds (see e.g. \cite{FK}, Theorem 2.2.1).

One important step in the characterization of compactness of the Green operator is the following observation
which we present separately for later use:

\begin{lem}\label{lem:i}
Let $V_1, V_2, V_3$ as in Lemma \ref{lem:Q},
assume that the G\r{a}rding inequality \eqref{eq:garding1} is satisfied on open subsets $\Omega\subset\subset X$
for all $u\in C^\infty_{cpt}(\Omega)\cap V_2$, and that $V_2$ is closed under multiplication with smooth compactly supported functions.

Let $k>0$,
$$\|u\|_G^2 = \|u\|^2_{L^2} + \|Tu\|^2_{L^2} + \|T^* u\|^2_{L^2}$$
and
$$\mathcal{K} = \{u\in V_2: u\in\Dom(T)\cap\Dom(T^*), \|u\|^2_G<k\}.$$
Then: for all $\epsilon>0$ and all $\Omega\subset \subset X$, there exists $\delta>0$ such that
\begin{eqnarray*}
\|\Phi^* u -u\|_{L^2} <\epsilon
\end{eqnarray*}
for all $u\in\mathcal{K}$ and all $\Phi\in\Def_\delta(\Omega,X)$, i.e. $\mathcal{K}$ satisfies the first
condition of the criterion Theorem \ref{thm:precompact}.
\end{lem}

\begin{proof}
Let $\epsilon>0$ and $\Omega\subset\subset X$.
Fix $\Omega\subset\subset \Omega_1 \subset\subset\Omega_2 \subset\subset X$
and $\chi\in C^\infty_{cpt}(\Omega_2,\R)$, $0\leq\psi\leq 1$, a cut-off function such that $\chi\equiv 1$ on $\Omega_1$.
Then
\begin{eqnarray}\label{eq:equal}
\|\Phi^* u - u\|_{L^2} = \|\Phi^* (\chi u) -\chi u\|_{L^2}
\end{eqnarray}
since $\Phi|_{X-\Omega}$ is just the identity mapping for all $\Phi\in \Def_\delta(\Omega,X)$.
Since multiplication with $\chi$ preserves $V_2$ and the domains of $T$ and $T^*$,
$$\chi u \in\Dom(T)\cap \Dom(T^*)\cap V_2$$
and there exists a constant $C_\chi>0$ such that
\begin{eqnarray}\label{eq:chi}
\|\chi u\|^2_{L^2} + \|T(\chi u)\|^2_{L^2} + \|T^*(\chi u)\|^2_{L^2} = \|\chi u\|_G^2 \leq C_\chi k
\end{eqnarray}
for all $u\in\mathcal{K}$.
By the same argument, we can use a partition of unity subordinate to a finite covering of $\Omega_2$
by coordinate charts to achieve that the $\chi u$ are supported in coordinate charts.
So, we can assume that $\Omega_2$ is a bounded domain in $\C^n$ (taking Lemma \ref{lem:trafo} into account).

Since $\chi u$ has compact support in $\Omega_2$, it can be approximated by smooth compactly supported forms in the $L^2$-sense
such that arbitrary partial derivatives (up to a certain order) converge as well in the $L^2$-sense
(making the $\|\cdot\|_G$-norm converge).
So, we can assume that the $\chi u$ are smooth with compact support in $\Omega_2$ because on the other hand
$$\|\Phi^* u -\Phi^* \wt{u}\|_{L^2} \leq 2 \|u-\wt{u}\|_{L^2}$$
if only $\delta<\delta_0(\Omega)$ for some fix $\delta_0(\Omega)>0$ (see \eqref{eq:def1} in the proof of Lemma \ref{lem:deformation1}).

To simplify the notation, let $v=\chi u$. 
As $|\frac{\partial}{\partial t}\Phi_t| \leq 3 \delta$
for all $\Phi\in \Def_\delta(\Omega,X)$ by Definition \ref{defn:Def},
\begin{eqnarray*}
|\Phi^* v(z) - v(z)| &=& |v(\Phi_1(z)) - v(\Phi_0(z))|\\
&\leq& \int_0^1 |v|_1 (\Phi_t(z)) \left|\frac{\partial}{\partial t} \Phi_t(z)\right|dt\\
&\leq& 3 \delta \int_0^1 |v|_1 (\Phi_t(z)) dt,
\end{eqnarray*}
where $|v|_1$ denotes the pointwise norm of all derivatives of first order of all coefficients of $v$.
Since $\|(\Phi_t)_*\|_\infty, \|(\Phi_t^{-1})_*\|_\infty <1+3\delta$, it follows as in the proof of Lemma \ref{lem:deformation1}
that
\begin{eqnarray*}
\|\Phi^*v -v \|^2_{L^2} &\leq & 9 \delta^2 \int_{\Omega} \left(\int_0^1 |v|_1 (\Phi_t(z)) dt\right)^2 e^{-\varphi(z)} dV_X(z)\\
&\leq& 9 \delta^2 \int_0^1 \left( \int_{\Omega} \Phi_t^* |v|_1^2 e^{-\varphi} dV_X \right) dt\\
&\leq& 9 \delta^2 \int_0^1 (1+3\delta)^{3n} \sup_{z\in\Omega}\big| e^{\varphi(z) - \varphi(\Phi_t^{-1}(z))}\big| \int_{\Omega} |v|_1^2 e^{-\varphi} dV_X dt.
\end{eqnarray*}
Since $\varphi\in C^0(X)$ is uniformly continuous on compact subsets of $X$,
there exists a constant
$$C_\varphi = \sup_{t\in[0,1]} \sup_{z\in\Omega_2} \big| e^{\varphi(z) - \varphi(\Phi_t^{-1}(z))}\big| < \infty.$$
So, we get
\begin{eqnarray*}
\|\Phi^* v -v \|^2_{L^2} \leq 9 \delta^2 (1+3\delta)^{3n} C_\varphi \|v\|^2_{W^{1,2}(\Omega,\varphi)}.
\end{eqnarray*}
It follows with \eqref{eq:equal} that there exists a constant $C(\Omega,\chi,\varphi)>0$
such that
\begin{eqnarray*}
\|\Phi^* u -u \|^2_{L^2} = \|\Phi^*(\chi u) - \chi u\|^2_{L^2}
\leq C(\Omega,\chi,\varphi) \delta^2 \|\chi u\|^2_{W^{1,2}(\Omega_2,\varphi)}
\end{eqnarray*}
for all $u\in \mathcal{K}$ and all $\Phi\in Def_\delta(\Omega,X)$.
This clearly is now the place to use the assumption that $T\oplus T^*$ is elliptic
in the sense of \eqref{eq:garding1}. Recall that we can assume that $\chi u$ is smooth with compact support
in $\Omega_2$. Hence there exists a constant $C_{\Omega_2}>0$ such that
\begin{eqnarray*}
\|\chi u\|^2_{W^{1,2}(\Omega_2,\varphi)} &\leq& C_{\Omega_2}
 \left( \|\chi u\|^2_{L^2_*(\Omega_2,\varphi)} + \|T(\chi u)\|^2_{L^2_*(\Omega_2,\varphi)} + \|T^* (\chi u)\|^2_{L^2_*(\Omega_2,\varphi)}\right)\\
&=& C_{\Omega_2} \|\chi u\|^2_G
\end{eqnarray*}
for all $u\in\mathcal{K}$. With \eqref{eq:chi}, we arrive finally at
$$\|\Phi^* u -u \|^2_{L^2} \leq \delta^2 C(\Omega,\chi,\varphi) C_{\Omega_2} C_\chi k.$$
Therefore,
$\|\Phi^* u -u\|^2_{L^2} <\epsilon$
for all $u\in\mathcal{K}$ and all $\Phi\in Def_\delta(\Omega,X)$ if $\delta$ is small enough.
\end{proof}

It is now easy to give a necessary and sufficient condition for compactness of the Green operator.
The criterion is inspired by the work of Gansberger \cite{G} who treats domains in $\C^n$. 
Part of his criterion goes back to an earlier work of Haslinger (see \cite{H}).

We restrict our attention to a Hermitian submanifold $X$ of a Hermitian complex space $Z$
in order to get an easy treatable notion of the boundary $bX$ of $X$.

\begin{thm}\label{thm:compact}
Let $Z$ be a Hermitian complex space,
$X\subset Z$ an open Hermitian submanifold in $Z$, and $T$ a linear partial differential operator
acting on $\Dom(T)\subset L^2_*(X,\varphi) \rightarrow L^2_*(X,\varphi)$ 
which is densely defined, closed, elliptic in the interior of $X$ and satisfies $T^2=0$.
By ellipticity, we understand that the G\r{a}rding inequality \eqref{eq:garding1} holds on each relatively
compact subset of $X$. Let
$$P= T^*T + TT^*.$$
Assume that there are closed subspaces $V_1, V_2, V_3 \subset L^2_*(X,\varphi)$
such that the assumptions of Lemma \ref{lem:Q} are satisfied, hence
$$Q=P|_{V_2}: \Dom Q\subset V_2 \rightarrow V_2$$
is self-adjoint with closed range, and let
$$D=\{u\in V_2: u\in \Dom(T)\cap\Dom(T^*)\cap\Ra(Q)\}.$$
Assume that $V_2$ is closed under multiplication with smooth compactly supported functions.
Then the following conditions are equivalent:

(i) The Green operator $Q^{-1}: V_2 \rightarrow V_2$ is compact.

(ii) The injection $j$ of $D$ equipped with the graph norm $\|u\|_D^2=\|Tu\|^2_{L^2}+\|T^*u\|^2_{L^2}$
into $L^2_*(X,\varphi)$ is compact.

(iii) For all $\epsilon>0$, there exists $\Omega\subset\subset X$ such that $\|u\|_{L^2_*(X-\Omega,\varphi)}<\epsilon$
for all $u\in \mathcal{L}=\{u\in D: \|u\|_D<1\}.$

(iv) There exists  a smooth function $\psi\in C^\infty(X,\R)$, $\psi>0$, such that $\psi(z)\rightarrow \infty$ as $z\rightarrow bX$,
and
\begin{eqnarray*}
( Qu,u)_{L^2} \geq \int_X \psi |u|^2 e^{-\varphi} dV_X\ \ \mbox{ for all } u\in\Dom(Q)\cap\Ra(Q).
\end{eqnarray*}
\end{thm}

\begin{proof}
First, we observe that (i) is equivalent to (ii). Since
$$Q^{-1}=j\circ j^*: V_2 \rightarrow V_2$$
by \eqref{eq:Q-1}, the assertion descends to the fact that a bounded operator
$S$ is compact exactly if $S^*$ is compact (see \cite{Rd}, Theorem 4.19), 
and $SS^*$ is compact exactly if $S$ and $S^*$ are compact (use $(S^*Sx,x)=(Sx,Sx)$).\\

We will now show that (ii) $\Rightarrow$ (iii) $\Rightarrow$ (iv) $\Rightarrow$ (ii). \newline
Assume that $j: D \hookrightarrow L^2_*(X,\varphi)$ is compact.
Then $j(\mathcal{L})$ is precompact in $L^2_*(X,\varphi)$ and (iii) holds by Theorem \ref{thm:precompact}.

If (iii) holds, it follows by linearity of $T\oplus T^*$ that for all $\epsilon>0$ there exists a domain
$\Omega_\epsilon\subset\subset X$ such that
\begin{eqnarray}\label{eq:ii1}
\|u\|^2_{L^2_*(X-\Omega_\epsilon,\varphi)} \leq \epsilon^2 \|u\|_D^2
\end{eqnarray}
for all $u\in D$. For such $u$, we have
\begin{eqnarray}\label{eq:ii2}
\int_X |u|^2 e^{-\varphi} dV_X \leq c^{-1} \|u\|^2_D
\end{eqnarray}
by \eqref{eq:apriori}, and for $k\in\N$, $k\geq 1$:
\begin{eqnarray}\label{eq:ii3}
\int_{X - \Omega_{2^{-k}}} 2^k |u|^2 e^{-\varphi} dV_X \leq 2^k \cdot 2^{-2k} \|u\|^2_D = 2^{-k} \|u\|^2_D
\end{eqnarray}
by \eqref{eq:ii1}. So, let $\psi'\in C^\infty(X,\R)$ be a real-valued smooth function such that
$$\left.\begin{array}{l}
\psi' \leq 2^{k}\\
\psi' \geq 2^{k-1}
\end{array}\right\} \mbox{ on } \Omega_{2^{-(k+1)}} - \Omega_{2^{-k}},\ k\geq 0,$$
where we set $\Omega_1=\emptyset$. It follows with \eqref{eq:ii2} and \eqref{eq:ii3} that
$$\int_X \psi' |u|^2 e^{-\varphi} dV_X \leq (c^{-1} + 1) \|u\|^2_D = (c^{-1} +1) (Qu,u)_{L^2}$$
for all $u\in\Dom(Q)\cap\Ra(Q)$. So, (iv) is satisfied with $\psi= (c^{-1}+1)^{-1}\psi'$.

It remains to show (iv) $\Rightarrow$ (ii). Assume that (iv) holds.
It is enough to show that $j(\mathcal{L})$ is precompact in $L^2_*(X,\varphi)$.
This will be done by checking the two conditions in Theorem \ref{thm:precompact}
for $j(\mathcal{L})$.

The second condition in Theorem \ref{thm:precompact} is obvious: Let $\epsilon>0$.
Choose $\Omega_\epsilon \subset\subset X$ such that $\psi \geq 1/\epsilon^2$ on $X-\Omega_\epsilon$.
Then
$$\epsilon^{-2} \int_{X-\Omega_\epsilon} |u|^2e^{-\varphi} dV_X \leq \int_X \psi |u|^2 e^{-\varphi} dV_X \leq \|u\|_D^2 \leq 1$$
for all $u\in \mathcal{L}\cap \Dom(Q)$. That proves the second condition as $\Dom(Q)$ is dense in $D$.
By Lemma \ref{lem:Q} \eqref{eq:apriori}, we have
\begin{eqnarray}\label{eq:bounded}
\|u\|^2_{L^2} \leq c^{-1} \left( \|Tu\|^2_{L^2} + \|T^*u\|^2_{L^2}\right) = c^{-1} \|u\|^2_D < c^{-1}
\end{eqnarray}
for all $u\in\mathcal{L}$. Hence, $\mathcal{L}$ is a subset of $\mathcal{K}$ 
in Lemma \ref{lem:i} with $k=1+c^{-1}$,
and so Lemma \ref{lem:i} yields the first condition in Theorem \ref{thm:precompact}.
\end{proof}

\newpage
If $T$ is a closed extension of the $\dq$-operator and $G^{-1}$ the $\dq$-Neumann operator
associated to this $\dq$-operator, Theorem \ref{thm:compact} reads as:

\begin{thm}\label{thm:compact2}
Let $Z$ be a Hermitian complex space of pure dimension $n$, $X\subset Z$ an open Hermitian submanifold
and $\dq$ a closed $L^2(X,\varphi)$-extension of the $\dq_{cpt}$-operator on smooth forms with compact support in $X$,
for example $\dq=\dq_w$ the $\dq$-operator in the sense of distributions.
Let $0\leq p,q \leq n$.

Assume that $\dq$ has closed range in $L^2_{p,q}(X,\varphi)$ and in $L^2_{p,q+1}(X,\varphi)$.
Then
$$\Box_{p,q} = \dq_{p,q} \dq^*_{p,q+1} + \dq^*_{p,q}\dq_{p,q-1}$$
has closed range and the following conditions are equivalent:

(i) The $\dq$-Neumann operator $N_{p,q}=\Box^{-1}_{p,q}$ is compact.

(ii) For all $\epsilon>0$, there exists $\Omega\subset\subset X$ such that $\|u\|_{L^2_{p,q}(X-\Omega,\varphi)}<\epsilon$
for all
$$u\in \{u\in \Dom(\dq)\cap\Dom(\dq^*)\cap \Ra(\Box_{p,q}): \|\dq u\|^2_{L^2} + \|\dq^*u\|^2_{L^2}<1\}.$$

(iii) There exists  a smooth function $\psi\in C^\infty(X,\R)$, $\psi>0$, such that $\psi(z)\rightarrow \infty$ as $z\rightarrow bX$,
and
\begin{eqnarray*}
( \Box_{p,q} u,u)_{L^2} \geq \int_X \psi |u|^2 e^{-\varphi} dV_X\ \ \mbox{ for all } u\in\Dom(\Box_{p,q})\cap\Ra(\Box_{p,q}).
\end{eqnarray*}
\end{thm}

\begin{proof}
The assumptions of Theorem \ref{thm:compact} are satisfied for
$$V_j = L^2_{p,q-2+j}(X,\varphi)\ \ \ ,\ j\in\{1,2,3\}.$$
We use $V_1=\{0\}$ if $q=0$, and $V_3=\{0\}$ if $q=n$.
\end{proof}

\newpage
\section{Compactness of the $\dq$-Neumann operator on complex spaces with isolated singularities}

\subsection{Compact solution operators for the $\dq_w$-equation}\label{sec:sop}

In this section, we use some $L^2$-regularity results for the $\dq_w$-equation at isolated singularities
due to Forn{\ae}ss, {\O}vrelid and Vassiliadou (see \cite{FOV2}) to construct compact solution
operators for the $\dq_w$-equation.

Let $X$ be a connected Hermitian complex space of pure dimension $n$ with only isolated singularities and
$\Omega\subset\subset X$ a relatively compact domain.
Assume that either $X$ is compact and $\Omega=X$, or that $X$ is Stein and
$\Omega$ has smooth strongly pseudoconvex boundary which does not contain singularities, $b\Omega\cap \Sing X=\emptyset$.

Let $\Omega^* = \Omega -\Sing X$ and $A=\Omega \cap \Sing X=\{a_1, ..., a_m\}$
the set of isolated singularities in $\Omega$.
For $z\in X$, we denote by $d_A(z)$ the distance $\dist_X(z,A)$ of the point $z$ to the singular set $A$ in $X$.
Here, $\dist_X(x,y)$ is the infimum of the length of piecewise smooth curves connecting
two points $x,y$ in $X$.

\begin{thm}\label{thm:fov2}
Let $X, \Omega, A$ as above and $p+q<n$, $q\geq 1$.
Then there exists a closed subspace $H$ of finite codimension in
$$\ker \dq_w: L^2_{p,q}(\Omega^*) \rightarrow L^2_{p,q+1}(\Omega^*)$$
and a constant $C>0$ such that for each $f\in H$ there exists $u\in L^2_{p,q-1}(\Omega^*)$
with $\dq_w u=f$ satisfying
\begin{eqnarray}\label{eq:est1}
\int_{\Omega^*} |u|^2 d_A^{-2} \log^{-4} (1 + d_A^{-1}) dV_X \leq C \int_{\Omega^*} |f|^2 dV_X.
\end{eqnarray}
For $p+q>n$, there exist constants $a>0$, $C_a>0$ and a closed subspace $L$ of finite codimension in
$$\ker \dq_w: L^2_{p,q}(\Omega^*) \rightarrow L^2_{p,q+1}(\Omega^*)$$
such that 
for each $f\in L$ there exists $u\in L^2_{p,q-1}(\Omega^*)$
with $\dq_w u=f$ satisfying
\begin{eqnarray}\label{eq:est2}
\int_{\Omega^*} |u|^2 d_A^{-2a} dV_X \leq C_a \int_{\Omega^*} |f|^2 dV_X.
\end{eqnarray}

If $X$ is Stein and $\Omega$ has smooth strongly pseudoconvex boundary which does
not contain singularities, then the $\dq_w$-equation is solvable with the estimate \eqref{eq:est2} if $p+q>n$, i.e. $L=\ker\dq_w$,
and $a>0$ can be chosen arbitrarily in $(0,1)$.

If $\Omega=X$ is compact, we have to assume in the second case (i.e. in the case $p+q>n$) that either $p+q>n+1$ or that $p=n$ and $q=1$.
\end{thm}

\begin{proof}
We will first treat the case that $X$ is Stein and $\Omega\subset\subset X$ has smooth strongly pseudoconvex boundary
that does not contain singularities.

We observe that there exists a strictly plurisubharmonic exhaustion function
$\rho'$ of $X$ which takes the value $\rho'=-\infty$ exactly on the singular set of $X$
and which is real-analytic outside.
This follows from \cite{CM}, Theorem 1.2,
and the observation that $M$ is a $1$-convex space with exceptional set $\pi^{-1}(\Sing X)$
if $\pi: M \rightarrow X$ is a resolution of singularities. 
We will explain desingularization in more details below.
Let $\rho=e^{\rho'}$.

After restricting $\rho$ to a neighborhood of $\o{\Omega}$,
there exists an arbitrarily small regular value $c>0$ of $\rho$ such that $\{\rho<c\}\subset\subset \Omega$.
Since $\Omega$ has smooth strongly pseudoconvex boundary, it can be exhausted by an increasing
sequence of smoothly bounded pseudoconvex domains.
So, all the assumptions of Proposition 5.8 and Theorem 5.9 in \cite{FOV2} are fulfilled.

The statement for $p+q<n$ follows
from the combination of Theorem 5.8 and Theorem 1.1 in \cite{FOV2}
by the following observation:
Let $a_j\in A$ be an isolated singularity. Then there exists a small neighborhood
$U_j$ of $a_j$ which can be embedded holomorphically in a complex number space $\C^L$
such that $a_j=0 \in \C^L$ and
$$\|z\| \lesssim d_A(z),$$
because the Euclidean distance of a point $z$ to the origin is less or equal to the length
of curves connecting $z$ to the origin in $X$, if the length of a curve is measured
with respect to the Euclidean metric. But the restriction of the Euclidean metric to $X$
is isometric to the original Hermitian metric of $X$.

So, if the equation $\dq_w u=f$ is solvable on $U_j-\{a_j\}$ according to Theorem 1.1 from \cite{FOV2},
then:
\begin{eqnarray*}
\int_{U_j-\{a_j\}} |u|^2 d_A^{-2} \log^{-4}(1+d_A^{-1}) dV_X &\lesssim& \int_{U_j-\{a_j\}} |u|^2 \|z\|^{-2} (-\log \|z\|^2)^{-4} dV_X\\
&\lesssim& \int_{U_j-\{a_j\}} |f|^2 dV_X.
\end{eqnarray*}
By \cite{FOV2}, Theorem 1.1, there are only finitely many obstructions to the equation
$\dq_w u=f$ on $U_j-\{a_j\}$ with that estimate. So, \cite{FOV2}, Proposition 5.8., yields
that there are only finitely many obstructions to solving the equation $\dq_w u=f$ on $\Omega^*$
with the estimate \eqref{eq:est1}.

If $p+q>n$, then the statement of our theorem is just Theorem 5.9 in \cite{FOV2},
$L=\ker\dq_w$, and $a>0$ can be chosen arbitrarily in $(0,1)$.\\

It remains to treat the case that $X$ is compact and $\Omega=X$.
This can be done by use of a desingularization.
Let
\begin{eqnarray*}
\pi: M \rightarrow X
\end{eqnarray*}
be a resolution of singularities (which exists due to Hironaka \cite{Hi}), i.e. a proper holomorphic surjection such that
\begin{eqnarray*}
\pi|_{M-E}: M-E \rightarrow X-\Sing X
\end{eqnarray*}
is biholomorphic, where $E=\pi^{-1}(\Sing X)$ is the exceptional set.
We can assume that $E$ is a divisor with only normal crossings,
i.e. the irreducible components of $E$ are regular and meet complex transversely.

For the topic of desingularization, we refer to
\cite{AHL}, \cite{BiMi} and \cite{Ha}.
Let
\begin{eqnarray*}
\gamma:= \pi^* h
\end{eqnarray*}
be the pullback of the Hermitian metric $h$ of $X$ to $M$.
$\gamma$ is positive semidefinite (a pseudo-metric) with degeneracy locus $E$.

We give $M$ the structure of a Hermitian manifold with a freely chosen (positive definite)
metric $\sigma$. Then $\gamma \lesssim \sigma$
and $\gamma \sim \sigma$ on compact subsets of $M-E$.

Let $\mathcal{L}^{p,q}_\sigma$ be the sheaf of germs of $(p,q)$-forms which are locally square-integrable
with respect to the metric $\sigma$ and which have a $\dq$-derivate in the sense of distributions which
is also square-integrable. Let $\mathcal{I}$ be the sheaf of ideals of the exceptional set $E$.
Let $k\in\Z$.
If $E$ is given in a point $x\in M$ as the zero set of a (germ of a) holomorphic function $f$,
then
$$(\mathcal{I}^k \mathcal{L}^{p,q}_\sigma)_x = \{ u: f^{-k} u \in (\mathcal{L}^{p,q}_\sigma)_x\}.$$
We have to use the weighted $\dq$-operator in the sense of distributions
$$(\dq_k)_x u_x:= f^k \dq_w ( f^{-k} u_x),$$
which coincides with the usual $\dq_w$-operator if $k\geq 0$.
We obtain fine resolutions
$$0 \rightarrow \mathcal{I}^k\Omega_M^p \hookrightarrow
\mathcal{I}^k \mathcal{L}^{p,0}_\sigma \xrightarrow{\ \dq_k\ }
\mathcal{I}^k \mathcal{L}^{p,1}_\sigma \xrightarrow{\ \dq_k\ }
\cdots \xrightarrow{\ \dq_k\ }
\mathcal{I}^k \mathcal{L}^{p,n}_\sigma \rightarrow 0,$$
where $\Omega^p_M$ is the sheaf of germs of holomorphic $p$-forms on $M$.
By the abstract Theorem of de Rham, this implies
$$H^{p,q}_{\sigma,k}(U):=
\frac{\mbox{ker }
(\dq_k: \mathcal{I}^k\mathcal{L}^{p,q}_\sigma(U) \rightarrow \mathcal{I}^k\mathcal{L}^{p,q+1}_\sigma(U))}
{\mbox{Im }
(\dq_k: \mathcal{I}^k\mathcal{L}^{p,q-1}_\sigma(U) \rightarrow \mathcal{I}^k\mathcal{L}^{p,q}_\sigma(U))}
\cong H^q(U,\mathcal{I}^k\Omega_M^p)$$
for open sets $U\subset M$.
We will use the well-known fact that
\begin{eqnarray}\label{eq:finite}
\dim H^{p,q}_{\sigma,k}(M) = \dim H^q(M,\mathcal{I}^k \Omega^p_M) < \infty
\end{eqnarray}
for all $0\leq p, q\leq n$ since $M$ is compact.

By \cite{Rp1}, Lemma 2.1, or \cite{FOV1}, Lemma 3.1, respectively,
there exists an integer $N\leq 0$ depending on $\pi: M\rightarrow X$ such that
\begin{eqnarray*}
\mathcal{L}^{p,q}_\gamma \subset \mathcal{I}^N\mathcal{L}^{p,q}_\sigma
\end{eqnarray*}
for all $0\leq p,q\leq n$, where $\mathcal{L}^{p,q}_\gamma$ is defined with respect to the pseudo-metric $\gamma$
analogously to $\mathcal{L}^{p,q}_\sigma$.
Note that the $\dq_N$-equation extends over the exceptional set by
the $\dq$-extension Theorem 3.2 in \cite{Rp1}.

Let
$$H^{p,q}_w(\Omega^*) := 
\frac{\ker (\dq_w: L^2_{p,q}(\Omega^*) \rightarrow L^2_{p,q+1}(\Omega^*))}
{\mbox{Im } (\dq_w: L^2_{p,q-1}(\Omega^*) \rightarrow L^2_{p,q}(\Omega^*))}.$$
We can now define a map
\begin{eqnarray}\label{eq:iso}
\Psi=\pi^*: H^{p,q}_w(\Omega^*) \rightarrow H^{p,q}_{\sigma,N}(M)
\end{eqnarray}
by the following observation:
Let $u\in L^2_{p,q-1}(\Omega^*)$ and $g\in L^2_{p,q}(\Omega^*)$ such that $\dq_w u=g$ on $\Omega^*$.
Then
\begin{eqnarray*}
\pi^* u &\in& \mathcal{L}^{p,q-1}_\gamma(M) \subset \mathcal{I}^N \mathcal{L}^{p,q-1}_\sigma(M),\\
\pi^* g &\in& \mathcal{L}^{p,q}_\gamma(M) \subset \mathcal{I}^N \mathcal{L}^{p,q}_\sigma(M),
\end{eqnarray*}
where we extend $\pi^* u$ and $\pi^* g$ trivially over the exceptional set $E$.
It is clear that $\dq_w \pi^* u= \pi^* g$ on $M-E$,
and it follows by use of the $\dq$-extension Theorem 3.2 in \cite{Rp1} that
$$\dq_N \pi^* u = \pi^* g$$
on $M$. So, $\Psi=\pi^*$ is a well-defined map on cohomology classes.
We will now show that $\Psi=\pi^*$ is injective if $p+q < n$.
Let $[f] \in H^{p,q}_w(\Omega^*)$ and assume that 
$$\pi^*[f]=0 \in H^{p,q}_{\sigma,N}(M),$$
i.e. there exists a form $u\in\mathcal{I}^N \mathcal{L}^{p,q-1}_\sigma(M)$
such that
$$\dq_N u = \pi^* f.$$
Let $\chi \in C^\infty_{cpt}(X^*)$, $0\leq \chi\leq 1$, be a cut-off function
such that $1-\chi$ is supported only in small neighborhoods $\{U_1, ..., U_m\}$ of the isolated singularities.
Let $U=\bigcup U_j$ and $K=X - U$. Then $u':=(\pi^*\chi)u$ has compact support in $M-E$,
$\dq_N u' = \dq_w u'$ on $M$, and
$$\dq_w u' = \pi^* f$$
on $\pi^{-1}(K)$ since $\chi\equiv 1$ in a neighborhood of $K$. Set
$$v:= \big( \pi|_{M-E}^{-1})^* u' \in L^2_{p,q-1}(\Omega^*).$$
Then
$$\dq_w v = f$$
on $K$ and $\dq_w v\equiv 0$ in a neighborhood of the isolated singularities.
Consider
$$f' := f - \dq_w v \in L^2_{p,q}(\Omega^*).$$
This form is $\dq_w$-closed and has support in $U=\bigcup U_j$.
If $p+q<n$, then the equation $\dq_w v' = f'$ is solvable in $U$
in the category of $L^2$-forms with compact support in $U$
such that the estimate
\begin{eqnarray}\label{eq:est30}
\int_{U_j-\{a_j\}} |v'|^2 d_A^{-2} \log^{-4}(1+d_A^{-1}) dV_X \lesssim \int_{U_j-\{a_j\}} |f'|^2 dV_X
\end{eqnarray}
holds by \cite{FOV2}, Proposition 3.1, if we assume that $U$ has been chosen appropriately.
Hence
$$f = \dq_w (v' + v),$$
which shows that $\Psi=\pi^*$ is injective if $p+q<n$, so that
$$\dim H^{p,q}_w(\Omega^*) < \infty$$
by use of \eqref{eq:finite} and \eqref{eq:iso}.
The $L^2$-norms of $v$, $\dq_w v$ and $f'$ can be dominated by the $L^2$-norm of $f$.
Then $v'$ satisfies the estimate \eqref{eq:est1},
and that is also clear for the form $v$ which has support in a fixed compact set with positive distance to the singular set $A$.
That proves the theorem if $\Omega=X$ is compact and $p+q<n$.\\

Let us finally consider the case that $\Omega=X$ is compact and $p+q>n$.
Here, we must distinguish between the cases $q=1$ and $p+q>n+1$.
Let $q=1$ which implies that $p=n$. We need an observation about the behavior of
$(n,0)$ and $(n,1)$-forms under the resolution of singularities $\pi: M\rightarrow X$.

Since $\sigma$ is positive definite and $\gamma$ is positive semi-definite,
there exists a continuous function $g\in C^0(M,\R)$ such that
$$dV_\gamma = g^2 dV_\sigma.$$
This yields $|g| |\omega|_\gamma  = |\omega|_\sigma$
if $\omega$ is an $(n,0)$-form, and
\begin{eqnarray*}
|\omega|_\sigma \leq |g||\omega|_\gamma
\end{eqnarray*}
if $\omega$ is a $(n,q)$-form, $0\leq q\leq n$. 
So, for an $(n,q)$ form $\omega$ on $M$:
\begin{eqnarray}\label{eq:l2est2}
\int |\omega|_\sigma^2 dV_\sigma \leq \int g^{2} |\omega|_\gamma^2 g^{-2} dV_\gamma = \int |\omega|^2_\gamma dV_\gamma.
\end{eqnarray}
Hence, there exists a natural mapping
$$\Psi=\pi^*: H^{n,q}_w(\Omega^*) \rightarrow H^{n,q}_{\sigma}(M),$$
which is an isomorphism by \cite{Rp8}, Theorem 1.5.
That shows that especially $H^{n,1}_w(\Omega^*)$ is finite-dimensional,
but we need some additional considerations to obtain also the estimate \eqref{eq:est2}.

As above, let $U=\bigcup U_j$ be a neighborhood of the isolated singularities
such that the $\dq_w$-equation is solvable for $(n,1)$-forms on $U$ with the estimate \eqref{eq:est2}, 
and let $\chi_1\in C^\infty(U)$, $0\leq \chi_1\leq 1$, be a cut-off function which is identically $1$
in a smaller neighborhood of the singular set $A$.

For $[f]\in H^{n,1}_w(\Omega^*)$, let $u\in L^2_{n,0}(U^*)$ be a solution on $U^*$
and set
$$f' := f - \dq(\chi_1 u) \in [f] \in H^{n,1}_w(\Omega^*),$$
where we extend $\chi_1 u$ trivially to $\Omega^*$.
Now, if $[f]=0$ in $H^{n,1}_w(\Omega^*)$ which is equivalent to $[\pi^* f']=0$ in $H^{n,1}_\sigma(M)$, then there exists $g\in L^{n,0}_\sigma(M)$
such that
$$\dq_w g = \pi^* f'$$
on $M$. But $\pi^* f'$ vanishes identically in a fixed neighborhood of the exceptional set $E$.
Hence $g$ is a holomorphic $n$-form in a fixed neighborhood of the exceptional set.
There, it is smooth and bounded and the sup-norm is bounded by the $L^2$-norm of $f'$, 
which in turn is bounded by the $L^2$-norm of $f$.

Let $\varphi$ be a fixed weight function that vanishes exactly of order $1$ along the
exceptional set $E$. Then $\varphi^{-(1-\epsilon)}g$ is square-integrable for a small $\epsilon>0$ and its $L^2$-norm
can be estimated by the $L^2$-norm of $f$.
Since $\mathcal{L}^{n,0}_\sigma(M)\cong L^2_{n,0}(\Omega^*)$, it follows by Lemma 3.1. in \cite{FOV1} that there exists an exponent $a>0$ such that
$$d_A^{-a} (\pi|_{M-E}^{-1})^* g \in L^2_{n,0}(\Omega^*).$$

Hence,
$$(\pi|_{M-E}^{-1})^* g + \chi_1 u$$
is the desired solution of the equation $\dq_w u=f$ which satisfies the estimate \eqref{eq:est2}
with that exponent $a>0$.\\

Finally, let $p+q>n+1$ which implies that $q\geq 2$.
Here, we proceed similar to the case $p+q<n$.
First, we define a map
$$\Psi: H^{p,q}_w(\Omega^*) \rightarrow H^{p,q}_{\sigma,N}(M)$$
where $N\geq 1$ is an integer such that
\begin{eqnarray}\label{eq:inclusion22}
\mathcal{I}^{N-1} \mathcal{L}^{p,q}_\sigma \subset \mathcal{L}^{p,q}_\gamma
\end{eqnarray}
for all $0\leq p,q\leq n$, which is true if only $N\geq 1$ is big enough by \cite{Rp1}, Lemma 2.1, or \cite{FOV1}, Lemma 3.1,
respectively.

We can define the map $\Psi$ as follows.
Let $[f]\in H^{p,q}_w(\Omega^*)$.
By solving the $\dq$-equation on the neighborhood $U$ of the singular set as above,
we can switch to the representative
$$f' = f - \dq(\chi_1 u) \in [f].$$
Since $f'$ has compact support away from the singular set, $\pi^* f' \in \mathcal{I}^{N}\mathcal{L}^{p,q}_\sigma(M)$,
and we can define
$$\Psi([f]) := [\pi^* f'] \in H^{p,q}_{\sigma,N}(M).$$
We need to show that this assignment is well-defined as a map on cohomology classes.
So, assume that
$$\dq_w g= f$$
on $\Omega^*$. Let $g'= g - \chi_1 u$. Then $\dq_w g'\equiv 0$ on the neighborhood $W$ of $A$ where $\chi_1\equiv 1$.
By shrinking $W$ appropriately, the $\dq_w$-equation is solvable on $W$ in the $L^2$-sense for $(p,q-1)$
if $p+q> n+1$. Hence, let $v\in L^2_{p,q-2}(W^*)$ such that
$$\dq v = g' = g -\chi_1 u,$$
and choose a cut-off function $\chi_2\in C^\infty_{cpt}(W)$, $0\leq \chi_2\leq 1$, which is identically $1$
in a smaller neighborhood of the singular set $A$.
Let 
$$g'' = g' - \dq_w( \chi_2 v) = g - \chi_1 u - \dq_w(\chi_2 v) \in L^2_{p,q-1}(\Omega^*).$$
Then
$$\dq_w g'' = \dq_w g - \dq_w(\chi_1 u) = f'$$
and $g''$ has compact support away from the singular set.
Hence
$$\dq_w (\pi^* g'') = \pi^* f',$$
so that $[\pi^* f']=0$ in $H^{p,q}_{\sigma,N}(M)$ and $\Psi$ is actually well-defined.

It is clear that $\Psi$ is injective because of the assumption \eqref{eq:inclusion22},
and we have arranged the index $N\geq 1$ such that a solution $\dq h=\pi^*f'$
on $M$ satisfies
$$\varphi^{-(1-\epsilon)} h \in L_\gamma^{p,q-1}(M)$$
as above. Hence, again
$$d_A^{-a} (\pi|_{M-E}^{-1})^* h \in L^2_{p,q-1}(\Omega^*)$$
with the exponent $a>0$ from above, and
$(\pi|_{M-E}^{-1})^* h + \chi_1 u$
is the desired solution of the equation $\dq_w u=f$ which satisfies the estimate \eqref{eq:est2}
with that exponent.
\end{proof}

\begin{cor}\label{cor:fov2}
Let $X, \Omega$ and $p, q$  as in Theorem \ref{thm:fov2}. 
Then the $\dq$-operator in the sense of distributions
$\dq_w: L^2_{p,q-1}(\Omega^*) \rightarrow L^2_{p,q}(\Omega^*)$
has closed range.
\end{cor}

We are now in the position to construct compact solution operators for the $\dq_w$-equation.
Let $\varphi$ be the weight
$$\varphi = - \log \left(d_A^{-2} \log^{-4} (1 + d_A^{-1})\right)$$
if $p+q < n$ or
$$\varphi = - \log d_A^{-2a}$$
if $p+q>n$. For $p+q\neq n$, $q\geq 1$, and $q=1$ if $\Omega$ is compact and $p+q=n+1$, let
$$T_1: L^2_{p,q-2}(\Omega^*) \rightarrow L^2_{p,q-1}(\Omega^*,\varphi)$$
and
$$T_2: L^2_{p,q-1}(\Omega^*,\varphi) \rightarrow L^2_{p,q}(\Omega^*)$$
be the $\dq$-operators in the sense of distributions (ignore $T_1$ if $q=1$).

$T_1$ and $T_2$ are closed densely defined operators, $T_2\circ T_1=0$ and
$T_2$ has closed range by Theorem \ref{thm:fov2}. 

So, the adjoint operators $T_1^*$ and $T_2^*$ are closed densely defined operators with
$T_1^*\circ T_2^*=0$ and $T_2^*$ has closed range. We can use the orthogonal decomposition
\begin{eqnarray}\label{eq:decomp}
L^2_{p,q-1}(\Omega^*,\varphi) = \ker T_2 \oplus \Ra (T_2^*)
\end{eqnarray}
to define a bounded solution operator for the $\dq_w$-equation as in Lemma \ref{lem:Q}.
Let $H\subset L^2_{p,q}(\Omega^*)$ be the closed subspace from Theorem \ref{thm:fov2}
if $p+q<n$ or $H=L\subset L^2_{p,q}(\Omega^*)$ if $p+q>n$. For $f\in H$ let $Sf$ be the uniquely
defined element $u \perp \ker T_2$ such that $\dq_w u=f$. So,
\begin{eqnarray}\label{eq:S}
S: H \subset L^2_{p,q}(\Omega^*) \rightarrow \Ra(T_2^*) \subset L^2_{p,q-1}(\Omega^*,\varphi)
\end{eqnarray}
is a bounded solution operator for the $\dq_w$-equation and it satisfies
\begin{eqnarray}\label{eq:T1S}
T_1^*\circ S=0.
\end{eqnarray}
Since $L^2_{p,q-1}(\Omega^*,\varphi)$ is contained in  $L^2_{p,q-1}(\Omega^*)$,
we can show by use of the criterion for precompactness Theorem \ref{thm:precompact}
as in the proof of Theorem \ref{thm:compact} that $S$ is compact as an operator to the latter space.

\begin{thm}\label{thm:sop}
Let $p+q\neq n$. For $q\geq 2$, and $p+q\neq n+1$ if $\Omega$ is compact,
the $\dq_w$-solution operator $S$ is compact as an operator
$$S: \Dom S=H \subset L^2_{p,q}(\Omega^*) \rightarrow L^2_{p,q-1}(\Omega^*).$$
For $q=1$, there exists a bounded operator $P_0: H \rightarrow L^2_{p,0}(\Omega^*)$ such that
$S-P_0$ is a compact $\dq_w$-solution operator
$$S-P_0: \Dom S=H \subset L^2_{p,1}(\Omega^*) \rightarrow L^2_{p,0}(\Omega^*).$$
\end{thm}

\begin{proof}
We will only treat the case that $\Omega$ is Stein with smooth strongly pseudoconvex boundary.
The compact case follows by the same arguments but is much easier because there is no boundary to consider.
Let
$$\mathcal{L} = \{f\in H: \|f\|_{L^2_{p,q}(\Omega^*)} <1\}.$$
We will show that $S(\mathcal{L})$ is precompact in $L^2_{p,q-1}(\Omega^*)$ if $q\geq 2$.
To do this, we have to treat the singular set $A$ and the strongly 
pseudoconvex boundary $b\Omega$ separately.
So let $\chi\in C^\infty_{cpt}(\Omega)$, $0\leq \chi\leq1$, be a smooth
cut-off function with compact support in $\Omega$ such that $\chi\equiv 1$
in a neighborhood of the singular set $A$.
Let us first show that
$$\mathcal{L}_1 := \{\chi S(f): f\in \mathcal{L}\}$$
is precompact in $L^2_{p,q-1}(\Omega^*)$ by use of the criterion Theorem \ref{thm:precompact} with $X=\Omega^*$.

Since $S$ is bounded as an operator to $L^2_{p,q-1}(\Omega^*,\varphi)$, there exists a constant $C_S>0$ such that
\begin{eqnarray*}
\|u\|_{L^2_{p,q-1}(\Omega^*,\varphi)} \leq C_S
\end{eqnarray*}
for all $u\in S(\mathcal{L})$.

Let $K=\supp\chi$, $K^* = K - A$.
Now then, let $\epsilon>0$.
Choose $\Omega_\epsilon \subset\subset \Omega^*$ such that
$$e^{-\varphi} \geq 1/\epsilon^2$$ 
on $K^*-\Omega_\epsilon$.
This is possible because $K^* -\Omega_\epsilon$ is a neighborhood of $A$ if $\Omega_\epsilon$ is big enough
and $e^{-\varphi(z)} \rightarrow +\infty$ as $z$ approaches the singular set $A$.
Then
\begin{eqnarray*}
\epsilon^{-2} \int_{\Omega^*-\Omega_\epsilon} |\chi u|^2 dV_X &=&  \epsilon^{-2} \int_{K^*-\Omega_\epsilon} |\chi u|^2 dV_X\\
&\leq&  \int_{K^*-\Omega_\epsilon} |\chi u|^2 e^{-\varphi} dV_X\\
&\leq& \int_{\Omega^*} |u|^2 e^{-\varphi} dV_X \leq C_S^2
\end{eqnarray*}
for all $u\in S(\mathcal{L})$. Hence
$$\|v\|_{L^2_{p,q-1}(\Omega^* - \Omega_\epsilon)} \leq \epsilon C_S$$
for all $v\in \mathcal{L}_1$.
That proves the second condition in Theorem \ref{thm:precompact},
it remains to show the first condition.
We can use Lemma \ref{lem:i} with $X=\Omega^*$ and
\begin{eqnarray*}
&T=T_1:& V_1=L^2_{p,q-2}(\Omega^*) \rightarrow V_2=L^2_{p,q-1}(\Omega^*,\varphi),\\
&T=T_2:& V_2=L^2_{p,q-1}(\Omega^*,\varphi) \rightarrow V_3=L^2_{p,q}(\Omega^*),
\end{eqnarray*}
because we can use different weight functions for forms of different degree in all our considerations above.

For $u\in S(\mathcal{L})$, we have
$$\|u\|^2_{L^2_{p,q-1}(\Omega^*,\varphi)} + \|T_2 u\|^2_{L^2_{p,q}(\Omega^*)} < C_S^2 + 1$$
and $T_1^* u=0$. Since $\chi$ is constant outside a compact subset of $\Omega^*$,
there exists a constant $C_\chi>0$ such that
\begin{eqnarray}\label{eq:bound1}
\|\chi u\|^2_{L^2_{p,q-1}(\Omega^*,\varphi)} + \|T_2 (\chi u)\|^2_{L^2_{p,q}(\Omega^*)} + \|T_1^* (\chi u)\|^2_{L^2_{p,q-2}(\Omega^*)} < C_\chi (C_S^2 +1).
\end{eqnarray}
So, we can use Lemma \ref{lem:i} with $k=C_\chi (C_S^2 +1)$ yielding $\mathcal{L}_1\subset \mathcal{K}$.
Hence, $\mathcal{L}_1$ satisfies also the first condition in Theorem \ref{thm:precompact}.\\

The second step is to show that
$$\mathcal{L}_2 = \{(1-\chi) S(f): f\in\mathcal{L}\}$$
is precompact in $L^2_{p,q-1}(\Omega^*)$. But this follows from well-known results since
$\Omega$ has smooth strongly pseudoconvex boundary and $(1-\chi)$ has support away from the singular set $A$.

Let $V$ be an open neighborhood of $A$ in $\Omega$ such that
$$N=\o{V} \subset\subset \{z\in \Omega: \chi(z)=1\} \subset\subset \Omega,$$ 
and let
\begin{eqnarray*}
\pi: M \rightarrow X
\end{eqnarray*}
be a resolution of singularities as in the proof of Theorem \ref{thm:fov2}.
Set $\Omega'= \pi^{-1}(\Omega)$ and $N'=\pi^{-1}(N)$.
Again, let
\begin{eqnarray*}
\gamma:= \pi^* h
\end{eqnarray*}
be the pullback of the Hermitian metric $h$ of $X$ to $M$ which
is positive semidefinite with degeneracy locus $E$.

As above, give $M$ the structure of a Hermitian manifold with a freely chosen (positive definite)
metric $\sigma$. Then
$$\gamma \lesssim \sigma$$
on a neighborhood of $\o{\Omega'}$ and
$$\gamma \sim \sigma$$
on $\Omega'-N'$ since the degeneracy locus $E$ of $\gamma$ is compactly contained in $\pi^{-1}(V)$.

Recall that
there exists a continuous function $g\in C^0(M,\R)$ such that
$$dV_\gamma = g^2 dV_\sigma.$$
This yields $|g| |\omega|_\gamma  = |\omega|_\sigma$
if $\omega$ is an $(n,0)$-form, and
\begin{eqnarray*}
|g| |\omega|_\gamma \leq |\omega|_\sigma
\end{eqnarray*}
if $\omega$ is a $(p,0)$-form, $0\leq p\leq n$. So, for a $(p,0)$ form $\omega$ on $M$:
\begin{eqnarray}\label{eq:l2est}
\int |\omega|_\gamma^2 dV_\gamma \leq \int g^{-2} |\omega|_\sigma^2 g^2 dV_\sigma = \int |\omega|^2_\sigma dV_\sigma.
\end{eqnarray}

\newpage
We can now show that $\mathcal{L}_2$ is precompact by use of the resolution $\pi: M\rightarrow X$
and well-known results about strictly pseudoconvex manifolds.

Since $\gamma=\pi^*h \sim\sigma$ on $\Omega'-N'$, we have
$$L^{2}_{p,q-1}(\Omega-N) \cong L^{2,\sigma}_{p,q-1}(\Omega'-N').$$
But the forms in $\mathcal{L}_2$ have support in $\o{\Omega}-N$.
So, it is enough to show that $\pi^*\mathcal{L}_2$ is precompact in $L^{2,\sigma}_{p,q-1}(\Omega')$.

As in \eqref{eq:bound1}, there exists a constant $C_\chi'>0$ such that
\begin{eqnarray*}
\|v\|^2_{L^2_{p,q-1}(\Omega^*,\varphi)} + \|T_2 v\|^2_{L^2_{p,q}(\Omega^*)} + \|T_1^* v\|^2_{L^2_{p,q-2}(\Omega^*)} < C_\chi'
\end{eqnarray*}
for all $v\in \mathcal{L}_2$. We can ignore the weight $\varphi$ since the forms in $\mathcal{L}_2$
have support away from the singular set $A$ and get a constant $C_\chi''>0$ such that
\begin{eqnarray*}
\|v\|^2_{L^2_{p,q-1}(\Omega^*)} + \|\dq_w v\|^2_{L^2_{p,q}(\Omega^*)} + \|\dq_w^* v\|^2_{L^2_{p,q-2}(\Omega^*)} < C_\chi''
\end{eqnarray*}
for all $v\in \mathcal{L}_2$. For the same reason we obtain
\begin{eqnarray}\label{eq:gnorm}
\|\pi^* v\|^2_{L^{2,\sigma}_{p,q-1}(\Omega')} + \|\dq_w \pi^* v\|^2_{L^{2,\sigma}_{p,q}(\Omega')} + \|\dq_w^* \pi^* v\|^2_{L^{2,\sigma}_{p,q-2}(\Omega')} < C_\chi'''.
\end{eqnarray}
Since $\Omega'$ is a relatively compact subset of $M$ with a smooth strongly pseudoconvex boundary,
Kohn's basic estimate yields
$$\|\pi^* v\|^2_{W^{1/2,2,\sigma}_{p,q-1}(\Omega')} < C_1$$
for all $v\in\mathcal{L}_2$ with some constant $C_1>0$ if $q\geq 2$.
In this setting, the embedding
$$W^{1/2,2,\sigma}_{p,q-1}(\Omega') \hookrightarrow L^{2,\sigma}_{p,q-1}(\Omega')$$
is compact by the Sobolev embedding theorem, and this shows that $\pi^*\mathcal{L}_2$
is a precompact subset of $L^2_{p,q-1}(\Omega')$ if $q\geq2$.\\

It remains to consider the case $q=1$.
Let 
$$\Pi_0: L^{2,\sigma}_{p,0}(\Omega') \rightarrow \ker \dq_w \subset L^{2,\sigma}_{p,0}(\Omega')$$
be the Bergman projection (the orthogonal projection onto $\ker\dq_w$).

We can now define the operator
$$P_0: H \rightarrow \ker \dq_w \subset L^2_{p,0}(\Omega^*)$$
as
$$P_0(f) := (\pi|_{\Omega'-E}^{-1})^* \circ \Pi_0 \circ \pi^* \big( (1-\chi) S(f)\big).$$
Since $\pi: \Omega'-E \rightarrow \Omega^*$ is biholomorphic,
it is clear that $\dq_w P_0(f)=0$,
so that $S-P_0$ remains a solution operator for the $\dq_w$-equation.

Since $(1-\chi)\equiv 0$ on $N$, it is clear that
$$f \mapsto \Pi_0 \circ \pi^* \big( (1-\chi) S(f)\big)$$
is a bounded map $H\rightarrow \ker\dq_w \subset L^{2,\sigma}_{p,0}(\Omega')$.

\newpage
On the other hand, \eqref{eq:l2est} yields (because $E$ is thin):
\begin{eqnarray*}
\|(\pi|_{\Omega'-E}^{-1})^* \omega\|_{L^2_{p,0}(\Omega^*)} = \|\omega\|_{L^{2,\gamma}_{p,0}(\Omega')} \leq \|\omega\|_{L^{2,\sigma}_{p,0}(\Omega')}
\end{eqnarray*}
Hence
\begin{eqnarray}\label{eq:inclusion}
(\pi|_{\Omega'-E}^{-1})^*: L^{2,\sigma}_{p,0}(\Omega') \rightarrow L^2_{p,0}(\Omega^*)
\end{eqnarray}
is bounded, and we see that $P_0$ is a bounded linear map.

It is now easy to see by Kohn's basic estimates that $(1-\chi)S-P_0$ is compact.
Because of \eqref{eq:inclusion}, it is enough to show that
\begin{eqnarray}\label{eq:toshow}
\pi^* \mathcal{L}_2 - \Pi_0 \big( \pi^* \mathcal{L}_2\big)
\end{eqnarray}
is precompact in $L^{2,\sigma}_{p,0}(\Omega')$.
But \eqref{eq:gnorm} implies that there exists a constant $C_2>0$ such that
$$\|\pi^* v - \Pi_0(\pi^* v)\|_{W^{1/2,2,\sigma}_{p,0}(\Omega')} < C_2$$
for all $v\in\mathcal{L}_2$ since $\Omega'$ is a domain with smooth strongly pseudoconvex boundary.
Hence, \eqref{eq:toshow} follows by the Sobolev embedding theorem.
\end{proof}

\vspace{3mm}
\subsection{Compactness of the $\dq_w$-Neumann operator}\label{sec:cpt}

We can now study the $\dq_w$-Neumann operator on spaces with isolated singularities.
Let $X, \Omega, A$ as in Theorem \ref{thm:fov2}.
Then
$$\dq_{cpt}: C^\infty_*(\Omega^*) \rightarrow C^\infty_*(\Omega^*)$$
is a densely defined operator on $L^2_*(\Omega^*)$.
In this section, we study the maximal closed extension,
i.e. the $\dq$-operator in the sense of distributions, which we denote by $\dq_w$.
Let
$$\Box = \dq_w \dq_w^* + \dq_w^*\dq_w.$$
By Theorem \ref{thm:self-adjoint}, $\Box$ is a densely defined closed self-adjoint operator with
$$(\Box u,u)_{L^2} \geq 0.$$
By Corollary \ref{cor:fov2}, the $\dq$-operator in the sense of distributions 
has closed range in $L^2_{p,q}(\Omega^*)$ if $p+q\neq n$.
If $\Omega=X$ is compact, we have to assume in addition that $q=1$ if $p+q=n+1$.
So, if $p+q\neq n-1, n$ (and $q=1$ if $p+q=n+1$ and $\Omega$ is compact), then
$$\Box_{p,q} = \Box|_{L^2_{p,q}}: L^2_{p,q}(\Omega^*) \rightarrow L^2_{p,q}(\Omega^*)$$
has closed range and we have the orthogonal decomposition
$$L^2_{p,q}(\Omega^*) = \ker \Box_{p,q} \oplus \Ra(\Box_{p,q})$$
by Lemma \ref{lem:Q}. Hence, the associated Green operator
$$N_{p,q} = \Box_{p,q}^{-1}: L^2_{p,q}(\Omega^*) \rightarrow \Dom(\Box) \subset L^2_{p,q}(\Omega^*)$$
is well-defined as in \eqref{eq:Green}. $N_{p,q}$ is called the $\dq_w$-Neumann operator.\\

We will now observe that $N_{p,q}$ is a compact operator (if $p+q\neq n-1,n$).
This is the case exactly if the equivalent conditions of Theorem \ref{thm:compact} are satisfied.
However, we do not use Theorem \ref{thm:compact} to verify compactness,
but a classical argument due to Hefer and Lieb relying on the existence of compact solution
operators (see \cite{HL}).

\begin{thm}\label{thm:neumanncpt}
Let $X$ be a Hermitian complex space of pure dimension $n$ with only isolated singularities,
and $\Omega\subset\subset X$ a relatively compact open subset 
such that either $\Omega=X$ is compact, or $X$ is Stein and
$\Omega$ has smooth strongly pseudoconvex boundary
that does not contain singularities.

Let $p+q\neq n-1, n$ and $q\geq 1$. If $\Omega=X$ is compact and $p+q=n+1$, let $q=1$.
Then the $\dq$-operator in the sense of distributions $\dq_w$
has closed range in $L^2_{p,q}(\Omega^*)$ and $L^2_{p,q+1}(\Omega^*)$
so that the $\dq_w$-Neumann operator
$$N_{p,q} = \Box_{p,q}^{-1}=(\dq_w\dq_w^*+\dq_w^*\dq_w)_{p,q}^{-1}: L^2_{p,q}(\Omega^*) \rightarrow \Dom\Box_{p,q}\subset L^2_{p,q}(\Omega^*)$$
is well-defined as in \eqref{eq:Green}. $N_{p,q}$ is compact.
\end{thm}

\begin{proof}
Only compactness remains to show.
By Theorem \ref{thm:fov2} and Theorem \ref{thm:sop},
there exist closed subspaces of finite codimension
\begin{eqnarray*}
H_q &\subset& \ker\dq_w \subset L^2_{p,q}(\Omega^*),\\
H_{q+1} &\subset& \ker\dq_w\subset L^2_{p,q+1}(\Omega^*),
\end{eqnarray*}
and compact linear operators
\begin{eqnarray*}
S_q: H_q &\rightarrow& \Dom\dq_w \subset L^2_{p,q-1}(\Omega^*),\\
S_{q+1}: H_{q+1} &\rightarrow& \Dom \dq_w\subset L^2_{p,q}(\Omega^*),
\end{eqnarray*}
such that $\dq_w S_q u=u$ and $\dq_w S_{q+1} u=u$.

Now then $N_{p,q}$ is compact by \cite{HL}, Theorem 3.1.
We may outline the short and elegant proof for convenience of the reader.
Let $U_q$ and $U_{q+1}$ be the orthogonal complements of $H_q$ and $H_{q+1}$
in $\Ra(\dq_w)$ in $L^2_{p,q}$ and $L^2_{p,q+1}$ with basis $f_{1,q}, ..., f_{r_q,q}$
and $f_{1,q+1}, ..., f_{r_{q+1},q+1}$, respectively.
Choose $u_{j,k}$ with $\dq_w u_{j,k} = f_{j,k}$ and define the operators $T_q$ and $T_{q+1}$
on $\Ra(\dq_w)$ in $L^2_{p,q}$ and $L^2_{p,q+1}$ by
$$T_k(f) = T_k\left( \sum_{j=1}^{r_k} \alpha_j f_{j,k} + g\right) := \sum_{j=1}^{r_k} \alpha_j u_{j,k} + S_q(g)\ \mbox{ for } g\in H_k.$$
Then $T_k$, $k=q, q+1$, are compact linear solution operators for the $\dq_w$-operator on $\Ra(\dq_w)$.
Extend these operators to be zero on $\Ra(\dq_w)^\perp$.

For $k\in\{q,q+1\}$, let $P_k: L^2_{p,k-1}(\Omega^*) \rightarrow (\ker \dq_w)^\perp$ and $Q_k: L^2_{p,k}(\Omega^*) \rightarrow \Ra(\dq_w)$
be the orthogonal projections on these closed subspaces, and define
$$K_k := P_k T_k Q_k\ \ \mbox{ for } k=q,q+1.$$
Hefer and Lieb show that
$N_{p,q} = K_q^* K_q + K_{q+1} K_{q+1}^*$,
and that yields compactness of $N_{p,q}$ by compactness of $T_q, T_{q+1}$.
\end{proof}

\newpage
\subsection{Compactness of the $\dq_s$-Neumann operator}\label{sec:ds}

Another important operator is the minimal closed extension of
$\dq_{cpt}: C^\infty_*(\Omega^*) \rightarrow C^\infty_*(\Omega^*)$,
i.e. the closure of the graph in $L^2_*(\Omega^*)\times L^2_*(\Omega^*)$, which we denote by $\dq_s$.
A form $f\in L^2_{p,q}(\Omega^*)$ is in the domain of $\dq_s$ iff it is in the domain of $\dq_w$
and there exist a sequence $\{f_j\} \subset C^\infty_{p,q}(\Omega^*)$ such that $f_j\rightarrow f$ in $L^2_{p,q}(\Omega^*)$
and $\dq f_j \rightarrow \dq_w f$ in $L^2_{p,q+1}(\Omega^*)$.
The $\dq_s$-operator is dual to the $\dq_w$-operator in a sense we will elaborate now.
Note that
$$\dq_s = \dq_{cpt}^{**}$$
since it is the closure of the graph.
Let $*$ be the Hodge-$*$-operator on $\Omega^*$ (mapping $(p,q)$ to $(n-q,n-p)$-forms). Then
$$\theta_{cpt}= - \o{*} \dq_{cpt} \o{*}$$
is the formal adjoint of the $\dq$-operator (acting on smooth forms with compact support).
By definition,
$$\dq_w = \theta_{cpt}^*.$$
We also obtain the $\theta$-operator in the sense of distributions
$$\theta_w = \dq_{cpt}^* = -\o{*}\dq_w \o{*}$$
and the minimal closed extension
$$\theta_s = \theta_{cpt}^{**} = -\o{*} \dq_s \o{*}.$$

Thus, we obtain the duality relations
$$\dq_w^* = \theta_{cpt}^{**} = \theta_s = -\o{*}\dq_s \o{*}$$
and
$$\dq_s^* = \dq_{cpt}^{*} = \theta_w = -\o{*} \dq_w \o{*}.$$
Hence the $\dq_s$-Laplacian is related to the $\dq_w$-Laplacian as
$$\Box^s = \dq_s\dq_s^* + \dq_s^*\dq_s = \o{*} \Box \o{*},$$
and the $\dq_s$-Neumann operator $N^s$
is well-defined on $(n-p,n-q)$-forms exactly if the $\dq_w$-Neumann operator is well-defined on $(p,q)$-forms,
and in that case:
\begin{eqnarray*}
N^s_{n-p.n-q} = (\Box_{n-p,n-q}^s)^{-1} = \o{*}\Box_{p,q}^{-1} \o{*}= \o{*} N_{p,q} \o{*}.
\end{eqnarray*}

So, a direct consequence of Theorem \ref{thm:neumanncpt} is:

\begin{thm}\label{thm:neumanncpt2}
Let $X$, $\Omega$, $p$ and $q$ as in Theorem \ref{thm:neumanncpt}, and $a=n-p$, $b=n-q$.

Then the minimal closed extension $\dq_s$ of the $\dq$-operator
has closed range in $L^2_{a,b}(\Omega^*)$ and $L^2_{a,b+1}(\Omega^*)$
so that the $\dq_s$-Neumann operator
$$N^s_{a,b} = (\Box^s_{a,b})^{-1}=(\dq_s\dq_s^*+\dq_s^*\dq_s)_{a,b}^{-1}: L^2_{a,b}(\Omega^*) \rightarrow \Dom\Box_{a,b}\subset L^2_{a,b}(\Omega^*)$$
is well-defined as in \eqref{eq:Green}. $N^s_{a,b}$ is compact.
\end{thm}

\end{document}